\documentclass[preprint,nopreprintline,11pt]{elsarticle}
\usepackage[a4paper, left=2cm,right=2cm,top=2cm,bottom=2cm]{geometry}
\usepackage{amssymb}
\usepackage{mathrsfs}
\usepackage{amsmath,empheq,fixmath}
\usepackage{multirow}
\usepackage{array}
\usepackage{amsthm}
\usepackage{nicefrac,xfrac}

\usepackage{booktabs}
\usepackage{tabularray}

\usepackage{graphicx}

\usepackage{hyperref}       
\hypersetup{
	colorlinks=true,
	linkcolor={blue},
	citecolor={blue},
	urlcolor={blue},
	breaklinks=true,
	plainpages=true
}

\graphicspath{{images/}{../images/}}
\usepackage{subfiles}

\usepackage[justification=centering]{caption}
\usepackage{subcaption}
\usepackage{accents}
\usepackage{enumerate}

\usepackage{calc}
\usepackage{siunitx}
\usepackage{tikz} 
\usetikzlibrary{backgrounds}
\usetikzlibrary{shapes,arrows,positioning,intersections,quotes}
\usepackage{tcolorbox}
\usepackage{pgfplots}

\usepackage{mathptmx}

\pgfplotsset{compat=newest,scaled y ticks=true} 
\usepgfplotslibrary{units} 
\allowdisplaybreaks

\DeclareMathOperator*{\argmin}{arg\,min}
\theoremstyle{definition}

\newcommand{\colref}[2]{\hyperref[#2]{#1~\ref*{#2}}}
\newcommand{\remref}[1]{\colref{Remark}{#1}}
\newcommand{\figref}[1]{\colref{Figure}{#1}}
\newcommand{\secref}[1]{\colref{Section}{#1}}
\newcommand{\lemref}[1]{\colref{Lemma}{#1}}
\newcommand{\thmref}[1]{\colref{Theorem}{#1}}
\newcommand{\cororef}[1]{\colref{Corollary}{#1}}

\newtheorem{theorem}{\textbf{Theorem}}[section]
\newtheorem{corollary}{Corollary}[theorem]
\newtheorem{lemma}[theorem]{\textbf{Lemma}}

\newtheorem{remark}{\textbf{Remark}}

\newcommand{\email}[1]{\href{mailto:#1}{\nolinkurl{#1}}}

\newcommand{\uh}{u_h}
\newcommand{\vh}{v_h}

\newcommand{\er}{e}
\newcommand{\eh}{e_h}
\newcommand{\ei}{\hat{e}}
\newcommand{\spaceV}{V}
\newcommand{\spaceVh}{V_h}
\newcommand{\spaceVd}{V^D}
\newcommand{\spaceVhd}{V^D_h}
\newcommand{\spaceVdiv}{\mvec{V}_{\text{div}}}
\newcommand{\interpolant}{I_h}
\newcommand{\interp}[1]{\interpolant #1}
\newcommand{\nsd}{d}
\newcommand{\nsdt}{\hat{d}}
\newcommand{\xvec}{\mvec{x}}
\newcommand{\adv}{\mvec{a}}

\newcommand{\advst}{\widetilde{\mvec{a}}}
\newcommand{\bit}{\Gamma_{0}}
\newcommand{\bft}{\Gamma_{T}}
\newcommand{\bsp}{\Gamma_s}
\newcommand{\mesh}{K_h}
\newcommand{\edgeset}{\mathcal{E}_h}
\newcommand{\grad}{\nabla}
\newcommand{\gradst}{\widetilde{\nabla}}
\newcommand{\divergence}{\grad\cdot}
\newcommand{\laplacian}{\Delta}
\newcommand{\ellipticA}{K}
\newcommand{\ellipticAmat}{\mathbold{\ellipticA}}
\newcommand{\ellipticB}{A}
\newcommand{\ellipticBvec}{\mathbold{\ellipticB}}
\newcommand{\bform}{b}
\newcommand{\bformh}{\bform_h}
\newcommand{\lform}{l}
\newcommand{\lformh}{\lform_h}
\newcommand{\glspar}{\epsilon}
\newcommand{\glsparMax}{\glspar_{max}}
\newcommand{\normalvec}{\hat{\mathbold{n}}}

\newcommand{\mvec}[1]{{\mathbold{#1}}}
\newcommand{\mmat}[1]{{\mathbold{#1}}}

\newcommand{\intTimeSpace}{\int_0^T\int_{\spDom}}
\newcommand{\intSpaceTime}{\int_{\spDom}\int_0^T}
\newcommand{\intSpace}{\int_{\spDom}}
\newcommand{\intTime}{\int_0^T}

\newcommand{\modulus}[1]{| #1 |}
\newcommand{\inner}[2]{\left( #1, #2 \right)}
\newcommand{\innerh}[2]{\left( #1, #2 \right)_h}
\newcommand{\norm}[1]{\left\| #1 \right\|}
\newcommand{\normh}[1]{\left\| #1 \right\|_h}
\newcommand{\normL}[3][2]{\| #2 \|_{L^{#1}(#3)}}
\newcommand{\normH}[3][1]{\| #2 \|_{H^{#1}(#3)}}

\newcommand{\seminormH}[3][1]{| #2 |_{H^{#1}(#3)}}

\newcommand{\residual}[1]{\partial_t #1 + \adiv #1 - \visco \laplacian #1}
\newcommand{\adivst}{\advst \cdot \gradst}
\newcommand{\normVExpanded}[1]{\norm{#1}_{\bft}^2 + \visco \norm{\grad #1}^2 + \normh{\glspar^{\halfnice}[\residual{#1}]}^2}
\newcommand{\normFT}[1]{\left\| #1 \right\|_{\bft}}

\newcommand{\oneOver}[1]{\frac{1}{#1}}

\newcommand{\approxConstant}{C_{a}}

\newcommand{\shapeRegularConst}{C_h}
\newcommand{\poincareConstant}{C_p}

\newcommand{\temporalInterval}{I_T}
\newcommand{\spDom}{\Omega} 
\newcommand{\stDom}{U} 

\newcommand{\ut}{\partial_t u}
\newcommand{\uht}{\partial_t \uh}
\newcommand{\vht}{\partial_t \vh}

\newcommand{\opL}{L}
\newcommand{\opLT}{M}
\newcommand{\visco}{\nu}
\newcommand{\adiv}{(\adv\cdot\grad)}
\newcommand{\ddiv}[1]{#1\cdot\grad}
\newcommand{\half}{\frac{1}{2}}
\newcommand{\halfnice}{\nicefrac{1}{2}}
\newcommand{\elmsum}{\Xsum_{K \in \mesh}}
\newcommand{\edgesum}{\Xsum_{E \in \edgeset}}
\newcommand{\normElm}[1]{\left\| #1 \right\|_K}
\newcommand{\he}{h_K}
\newcommand{\jump}{\mathbb{J}}

\newcommand{\etaFull}{\eta_{_\stDom}}
\newcommand{\etaElm}{\eta_K}
\newcommand{\nrefine}{N_{ref}}

\newcommand{\vertiii}[1]{{\left\vert\kern-0.25ex\left\vert\kern-0.25ex\left\vert #1 \right\vert\kern-0.25ex\right\vert\kern-0.25ex\right\vert}}
\newcommand{\normVh}[1]{\vertiii{#1}}
\newcommand{\normVhstar}[1]{\vertiii{#1}_*}

\newcommand{\setbuilder}[1]{\left\{ #1 \right\}}

\newcommand{\petsc}{\textsc{PETSC}}
\newcommand{\globalMatrix}{\mmat{A}}

\DeclareFontFamily{U} {cmex}{}

\DeclareFontShape{U}{cmex}{m}{n}{
	<-6> cmex5
	<6-7> cmex6
	<7-8> cmex7
	<8-9> cmex8
	<9-10> cmex9
	<10-12> cmex10
	<12-> cmex12}{}

\DeclareSymbolFont{Xcmex} {U} {cmex}{m}{n}

\DeclareMathSymbol{\Xsum}{\mathop}{Xcmex}{80}

\newcommand{\logLogSlopeTriangle}[5]
{
	
	\pgfplotsextra
	{
		\pgfkeysgetvalue{/pgfplots/xmin}{\xmin}
		\pgfkeysgetvalue{/pgfplots/xmax}{\xmax}
		\pgfkeysgetvalue{/pgfplots/ymin}{\ymin}
		\pgfkeysgetvalue{/pgfplots/ymax}{\ymax}
		
		\pgfmathsetmacro{\xArel}{#1}
		\pgfmathsetmacro{\yArel}{#3}
		\pgfmathsetmacro{\xBrel}{#1-#2}
		\pgfmathsetmacro{\yBrel}{\yArel}
		\pgfmathsetmacro{\xCrel}{\xArel}
		
		\pgfmathsetmacro{\lnxB}{\xmin*(1-(#1-#2))+\xmax*(#1-#2)} 
		\pgfmathsetmacro{\lnxA}{\xmin*(1-#1)+\xmax*#1} 
		\pgfmathsetmacro{\lnyA}{\ymin*(1-#3)+\ymax*#3} 
		\pgfmathsetmacro{\lnyC}{\lnyA+#4*(\lnxA-\lnxB)}
		\pgfmathsetmacro{\yCrel}{\lnyC-\ymin)/(\ymax-\ymin)} 
		
		\coordinate (A) at (rel axis cs:\xArel,\yArel);
		\coordinate (B) at (rel axis cs:\xBrel,\yBrel);
		\coordinate (C) at (rel axis cs:\xCrel,\yCrel);
		
		\draw[#5]   (A)-- node[pos=0.5,anchor=north] {1}
		(B)-- 
		(C)-- node[pos=0.5,anchor=west] {#4}
		cycle;
	}
}

\newcommand{\logLogSlopeTriangleFlipped}[6]
{
	
	\pgfplotsextra
	{
		\pgfkeysgetvalue{/pgfplots/xmin}{\xmin}
		\pgfkeysgetvalue{/pgfplots/xmax}{\xmax}
		\pgfkeysgetvalue{/pgfplots/ymin}{\ymin}
		\pgfkeysgetvalue{/pgfplots/ymax}{\ymax}
		
		\pgfmathsetmacro{\xArel}{#1}
		\pgfmathsetmacro{\yArel}{#3}
		\pgfmathsetmacro{\xBrel}{#1-#2}
		\pgfmathsetmacro{\yBrel}{\yArel}
		\pgfmathsetmacro{\xCrel}{\xArel}
		
		\pgfmathsetmacro{\lnxB}{\xmin*(1-(#1-#2))+\xmax*(#1-#2)} 
		\pgfmathsetmacro{\lnxA}{\xmin*(1-#1)+\xmax*#1} 
		\pgfmathsetmacro{\lnyA}{\ymin*(1-#3)+\ymax*#3} 
		\pgfmathsetmacro{\lnyC}{\lnyA-#4*(\lnxA-\lnxB)}
		\pgfmathsetmacro{\yCrel}{\lnyC-\ymin)/(\ymax-\ymin)} 
		
		\coordinate (A) at (rel axis cs:\xArel,\yArel);
		\coordinate (B) at (rel axis cs:\xBrel,\yBrel);
		\coordinate (C) at (rel axis cs:\xCrel,\yCrel);
		
		\draw[#6]   (A)-- node[pos=0.5,anchor=south] {1}
		(B)-- 
		(C)-- node[pos=0.5,anchor=west] {#5}
		cycle;
	}
}

\newcommand{\logLogSlopeTriangleDecreasing}[6]
{
	
	\pgfplotsextra
	{
		\pgfkeysgetvalue{/pgfplots/xmin}{\xmin}
		\pgfkeysgetvalue{/pgfplots/xmax}{\xmax}
		\pgfkeysgetvalue{/pgfplots/ymin}{\ymin}
		\pgfkeysgetvalue{/pgfplots/ymax}{\ymax}
		
		\pgfmathsetmacro{\xArel}{#1}
		\pgfmathsetmacro{\yArel}{#3}
		\pgfmathsetmacro{\xBrel}{#1-#2}
		\pgfmathsetmacro{\yBrel}{\yArel}
		\pgfmathsetmacro{\xCrel}{\xBrel}
		
		\pgfmathsetmacro{\lnxB}{\xmin*(1-(#1-#2))+\xmax*(#1-#2)} 
		\pgfmathsetmacro{\lnxA}{\xmin*(1-#1)+\xmax*#1} 
		\pgfmathsetmacro{\lnyA}{\ymin*(1-#3)+\ymax*#3} 
		\pgfmathsetmacro{\lnyC}{\lnyA+#4*(\lnxA-\lnxB)}
		\pgfmathsetmacro{\yCrel}{\lnyC-\ymin)/(\ymax-\ymin)} 
		
		\coordinate (A) at (rel axis cs:\xArel,\yArel);
		\coordinate (B) at (rel axis cs:\xBrel,\yBrel);
		\coordinate (C) at (rel axis cs:\xCrel,\yCrel);
		
		\draw[#6]   (A)-- node[pos=0.5,anchor=north] {1}
		(B)-- node[pos=0.5,anchor=east] {#5}
		(C)-- 
		cycle;
	}
}

\begin{document}

\begin{frontmatter}
\title{\textbf{Space-time finite element analysis of the advection-diffusion equation using Galerkin/least-square stabilization}}

\author[ISUME]{Biswajit Khara \texorpdfstring{\corref{CORRAUTH}}{}} \emailauthor{bkhara@iastate.edu}{Biswajit Khara}
\author[ISUME]{Kumar Saurabh}
\author[CUTME]{Robert Dyja}
\author[ISUAER]{Anupam Sharma}
\author[ISUME]{Baskar Ganapathysubramanian \texorpdfstring{\corref{CORRAUTH}}{}}
\emailauthor{baskarg@iastate.edu}{Baskar Ganapathysubramanian}
\address[ISUME]{Department of Mechanical Engineering, Iowa State University, Ames, IA, USA}
\address[ISUAER]{Department of Aerospace Engineering, Iowa State University, Ames, IA, USA}
\address[CUTME]{Faculty of Mechanical Engineering and Computer Science, Czestochowa University of Technology, Czestochowa, Poland}
\cortext[CORRAUTH]{Corresponding author}


\begin{abstract}
We present a full space-time numerical solution of the advection-diffusion equation using a continuous Galerkin finite element method on conforming meshes. The Galerkin/least-square method is employed to ensure stability of the discrete variational problem. In the full space-time formulation, time is considered another dimension, and the time derivative is interpreted as an additional advection term of the field variable. We derive \textit{a priori} error estimates and illustrate spatio-temporal convergence with several numerical examples. We also derive \textit{a posteriori} error estimates, which coupled with adaptive space-time mesh refinement provide efficient and accurate solutions. The accuracy of the space-time solutions is illustrated against analytical solutions as well as against numerical solutions using a conventional time-marching algorithm.
\end{abstract}
\end{frontmatter}

\section{Introduction}\label{sec:introduction}

Numerically solving a transient (or evolution) problem characterized by a partial differential equation (PDE) requires that the continuous problem be discretized in space and time. The standard way to deal with this dual discretization, is to use a suitable time-marching algorithm coupled with some form of spatial discretization such as the finite difference method (FDM), the finite element method (FEM), the finite volume method (FVM) or the more recent isogeometric analysis (IgA). In some cases, the time marching algorithms themselves are based on finite difference methods and are used along with the above mentioned spatial discretizations. Generally speaking, these approaches consider the spatio-temporal domain over which the solution is desired as a product of a spatial domain with a temporal domain, with independent discretization and analysis of each of these components.


An alternative strategy to ``time-marching" is to  discretize and solve for the full ``space-time" domain together. Any combination of spatial and temporal discretization can be used -- for instance, finite difference schemes in both space and time \cite{horton1995space, lubich1987multi}; or FEM in space and FDM in time \cite{dyja2018parallel}; or FEM in both space and time \cite{langer2016space}. In particular, in the context of FEM,  a ``space-time" formulation refers to one where finite element formulation is used in both space and time, but time marching is not employed. A major appeal of formulating a problem in space-time is the possibility of improved parallel performance. The idea of parallelism in both space and time builds on a rich history of parallel time integration~\cite{Hackbusch:1985:PMM:4673.4714, lubich1987multi, horton1995space}. We refer to \cite{gander201550,vandewalle2013parallel} for a review of such methods.

The finite element community has a history of considering solutions to time dependent PDEs in space-time. The earliest references to space-time formulations go back to the mid 1980's. Babuska and co-workers~\cite{babuska1989h,babuvs1990h} developed $h-p$ versions of finite element method in space along with $p$ and $h-p$ versions of approximations in time for parabolic problems. Around the same time, Hughes and Hulbert formulated a space-time finite element method for elastodynamics problems \cite{hughes1988space} and general hyperbolic problems \cite{hulbert1990space} using time-discontinuous Galerkin method. Recently, there has been increasing interest in revisiting this problem given access to larger computational resources \cite{langer2016space, dyja2018parallel, ishii2019solving}. In addition, in the case where FEM is used to discretize the space, the time-marching algorithm can also be based on finite elements. This formulation also known as \textit{space-time formulation} is usually applied to a single time step. This approach has been successfully applied to a rich variety of applications \cite{tezduyar1992newI,tezduyar1992newII,takizawa2012space}. The current work, in contrast, explores solving for large space-time blocks. In this work, we tackle two key aspects associated with solving evolution equations in space-time on conforming space-time meshes -- stability and computational cost. We focus on a particular family of PDEs, specifically the time dependent advection-diffusion equation, with the time dependent diffusion equation as a special case (when the advection term goes to zero). 


\textbf{Stability of the discrete space-time formulation}: When solving parabolic equations through space-time methods, the question of stability of the ensuing discrete system becomes important. As mentioned earlier, in a sequential setting, the time derivative term is treated separately during temporal discretization. But when the advection-diffusion equation is formulated in space-time (i.e. time is considered another dimension, like the rest of the spatial dimensions), the time evolution term (first order derivative with respect to time) can be mathematically seen as an ``advection in time dimension" term and can be grouped with the other spatial first derivatives in the equation \cite{bank2017arbitrary}. This identification is mathematically consistent since all the first order derivatives, irrespective of whether they are spatial or temporal, have a sense of ``directionality" attached to them (the actual direction is determined by the sign of their coefficients). Mathematically then, the problem can be seen as a ``generalized advection-diffusion equation" in space-time, except, there is no diffusion term associated with the time dimension. This type of equation, when solved through the standard Galerkin method, can suffer from a lack of stability and end up with spurious oscillations \cite{brooks1982streamline, donea2003finite}. Andreev and Mollet \cite{andreev2012stability, mollet2014stability} analyzed the stability of space-time FEM discretizations of abstract linear parabolic evolution equations. Steinbach \cite{steinbach2015space} also analyzed the stability of the heat equation in space-time setting and derived error bounds using unstructured space-time finite elements. In 2017, Langer et al. \cite{langer2016space} used a time-upwind type Petrov-Galerkin basis function similar to the streamline diffusion method, to solve the heat equation in space-time moving domains. There is a large body of literature available that deals with treating non self-adjoint operators through finite element method \cite{brooks1982streamline,johnson1984finite,johnson1986streamline,hughes1989new,hughes1995multiscale}, a review of which can be found in \cite{franca2004stabilized}. In this work, we build upon this body of work and ensure stability of the discrete variational form by using a Galerkin/least squares (GLS) approach \cite{hughes1989new, franca1992stabilized}. We show that GLS provides stability to the discrete space-time problem and derive \textit{a priori} error bounds in the discrete norm associated with the bilinear form.

\textbf{Reducing computational cost via space-time adaptivity}: A space-time formulation adds one more dimension to concurrently discretize. A 2D problem needs a 3D mesh and a 3D problem needs a 4D mesh. Thus the number of degrees of freedom in the resulting linear system can become significantly larger than the corresponding sequential problem. 
Prior research has shown that this increased computational cost is ideally suited to sustained parallelism \cite{horton1995space,farhat2003time,cortial2009time,friedhoff2012multigrid,emmett2012toward,speck2012massively,langer2016space,dyja2018parallel,ishii2019solving}. Additional efficiencies can be accessed by making efficient use of adaptive mesh refinement (AMR) in space-time \cite{abedi2004spacetime,dyja2018parallel,ishii2019solving,christopher2020parallel,gomez2023design,steinbach20197}. To this end, we mathematically derive a residual-based \textit{a posteriori} error indicator which can be used to estimate elementwise errors. This enables us to leverage the benefits of AMR in a space-time setting which render very accurate solution to the diffusion problem and the advection-diffusion problems considered here. 

The rest of the content of this paper is organized as follows. In \secref{sec:math-formulation}, we introduce the mathematical formulation of the problem. We formulate the continuous and the discrete problems and state the relevant function spaces. We then prove the stability of the discrete bilinear form and then derive \textit{a priori} and \textit{a posteriori} error estimates. In \secref{sec:implementation}, we briefly discuss some of the implementation details. Then, in \secref{sec:numerical-examples}, we present numerical examples that validate the theoretical results of \secref{sec:math-formulation}, and also demonstrate the advantages of using adaptive mesh refinement in space-time analysis. Finally, in \secref{sec:conclusion}, we draw conclusions and make some comments about future research directions.

\section{Mathematical Formulation}\label{sec:math-formulation}
\subsection{The time-dependent linear advection-diffusion equation}\label{subsec:pde}
\begin{figure}[!t]
	\centering
	\def \mpagefrac{0.7}
	\begin{minipage}{\mpagefrac\textwidth}
		\centering
		\begin{tikzpicture}
		\def \myscale{\mpagefrac / 0.5} 
		\draw (0, 0) node[inner sep=0] {\includegraphics[trim=3.3in 0.5in 3.3in 0.8in, clip,width=0.7\linewidth]{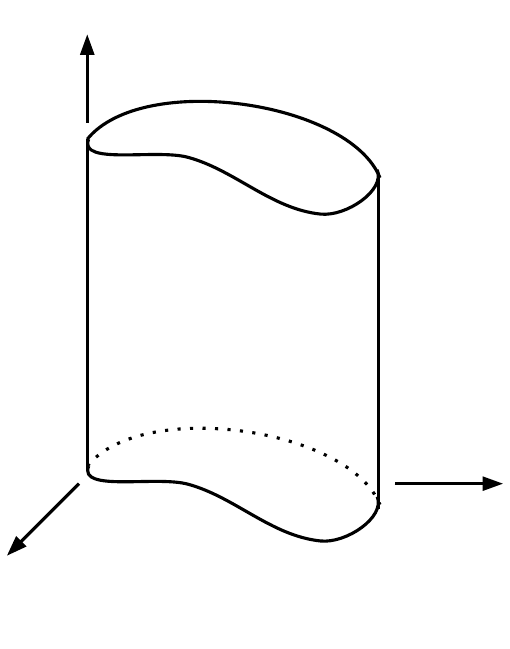}};
		\draw (0, 0) node {$ \stDom = \spDom \times \temporalInterval $};
		\draw (-2.5*\myscale, 3.0*\myscale) node {$t$};
		\draw (-2.7*\myscale,  2.1*\myscale) node {$t=T$};
		\draw (-2.7*\myscale, -1.5*\myscale) node {$t=0$};
		\draw (-2.2*\myscale, -2.7*\myscale) node {$x_1$};
		\draw ( 2.6*\myscale, -1.5*\myscale) node {$x_2$};
		\draw[stealth-, thick] (-0.2*\myscale, -1.6*\myscale) -- ++(0.5*\myscale,-1.3*\myscale) node[right, black, fill=white, anchor=north]{$\bit = \spDom\times \{0\}$};
		\draw[stealth-, thick] (-0.2*\myscale, 2.1*\myscale) -- ++(0.2*\myscale,1.0*\myscale) node[right, black, fill=white, anchor=south]{$\bft = \spDom\times \{T\}$};
		\draw[stealth-, thick] (1.45*\myscale, 0.1*\myscale) -- ++(0.7*\myscale,1.0*\myscale) node[right, black, fill=white, anchor=south west]{$\bsp = \partial \spDom \times (0,T)$};
		\end{tikzpicture}
	\end{minipage}
	\caption{Schematic depiction of the space-time domain $ \stDom = \spDom \times \temporalInterval $, where $ \spDom \subset \mathbb{R}^{\nsd} $, and $ \temporalInterval = [0,T] \subset \mathbb{R}^{+}. $} 
	\label{fig:schematic-space-time-domain}
\end{figure}
Consider a bounded spatial domain $\spDom \in \mathbb{R}^{\nsd}, \nsd=1,2,3$ with Lipschitz continuous boundary $\partial \spDom$ and a bounded time interval $\temporalInterval = (0,T] \in \mathbb{R^+}$. We define the space-time domain as the Cartesian product of the two as $\stDom = \spDom \times \temporalInterval = \spDom\times(0,T] \subset \mathbb{R}^{\nsdt}$, $\nsdt = \nsd + 1$ (see \figref{fig:schematic-space-time-domain}). The overall boundary of this space-time domain is defined as $\Gamma = \partial \stDom$. This overall boundary is the union of the spatial boundaries and the time boundaries. The spatial domain boundary is denoted by $\bsp = \partial \spDom \times (0,T]$; whereas the time boundaries are denoted by $\Gamma_{_0} = \bar{\spDom}\times\{0\}$ and $\Gamma_{_T} = \bar{\spDom}\times\{T\}$ which are the initial and final time boundaries respectively. The closure of the space-time domain is $ \bar{\stDom} = \stDom \cup \bsp \cup \Gamma_{_0} $. The advection-diffusion equation can then be written for the scalar function $u: \stDom \rightarrow \mathbb{R}$ as:
\begin{subequations}\label{pde:linear-advection-diffusion}
\begin{align}
    \ut + \adiv u- \visco \laplacian u &=f \quad \textrm{in}\quad\stDom\label{pde:linear-advection-diffusion-equation},\\
u&=g\quad \textrm{on}\quad\bsp\label{pde:linear-advection-diffusion-BC},\\
u&=u_0\quad \textrm{on}\quad \Gamma_{0},\label{pde:linear-advection-diffusion-IC}
\end{align}
\end{subequations}
where $f: \stDom \rightarrow \mathbb{R}$ is a smooth forcing function,  $\visco >0 $ is the diffusivity and does not depend on $u$; and Dirichlet boundary conditions are imposed on the boundary $\bsp$. Note that $\grad$ is the usual gradient operator in the space $\mathbb{R}^{\nsd}$, i.e., $ \grad \equiv (\nicefrac{\partial}{\partial x}, \nicefrac{\partial}{\partial y}) $ for $ \nsd = 2 $ and $ \grad \equiv (\nicefrac{\partial}{\partial x}, \nicefrac{\partial}{\partial y}, \nicefrac{\partial}{\partial z}) $ for $ \nsd = 3 $ respectively. We further define the space-time gradient operator $\gradst$ as: $ \gradst \equiv (\grad, \nicefrac{\partial}{\partial t}) $, and the space-time advection operator as $\advst = (\adv, 1)$ where the unity is the ``advection in time.''

Let us define two operators $ \opL $ and $ \opLT $ as follows:
\begin{subequations}
	\begin{align*}
	\opL u &:= \adiv u- \visco \laplacian u, \\
	\text{and}\ \ \opLT u &:= \ut + \opL u \\
	&= \ut + \adiv u- \visco \laplacian u.
	\end{align*}
\end{subequations}

Since $ \visco > 0$, therefore $ \opL $ is elliptic and it follows that the operator $\opLT = (\nicefrac{\partial}{\partial t} + L)$ is strictly parabolic. Such equations are typically solved with a \textit{method of lines} discretization, by solving a series of discrete equations sequentially. At each ``time-step'', the solution is assumed to be a function of the spatial variables only. In the context of the Galerkin methods, this spatial approximation takes the form of a (discrete) Sobolev space, e.g., the solution at the $ i^{\text{th}} $ step, $ u^i(x) \in  H^1(\spDom) $. But in this work, our focus will be \textit{coupled space-time} formulation of the advection-diffusion equation where we approximate the solution in a Sobolev space defined on the full space-time domain (i.e., $ u(x,t) \in H^1(\stDom) $). As alluded to in the introduction, this kind of problem formulation requires some form of stabilization. Here, the Galerkin/Least Squares method is used for this purpose.

When we add such stabilizing terms to the original variational problem, we essentially end up adding some numerical ``diffusion" into the system. More importantly, diffusion may be introduced in the time direction as well, which is absent in the original parabolic equation. We show, in the results section, that careful design of stabilization ensures that this diffusion in the time direction can be made arbitrarily small. In essence, then, the equation becomes elliptic in the space-time domain. Anticipating such changes to the equation, we can cast this parabolic equation as a generalized elliptic equation as follows,
\begin{align}\label{pde:generalized-elliptic-summation-form}
    -\Xsum_{i=1}^{\nsd+1}\Xsum_{j=1}^{\nsd+1} \frac{\partial}{\partial x_i}\left( \ellipticA_{ij} (x) \frac{\partial u}{\partial x_j}\right) + \Xsum_{i=1}^{\nsd+1} \ellipticB_i(x)\frac{\partial u}{\partial x_i} = f,
\end{align}
where $\ellipticA_{ij}$ and $\ellipticB_i$ are the components of $\ellipticAmat\in \mathbb{R}^{\nsdt \times \nsdt}$ and $\ellipticBvec\in \mathbb{R}^{\nsdt}$:
\begin{align} \label{expr:ellipticity-matrix-original}
  \ellipticAmat:= 
    \begin{bmatrix}
    \visco \mvec{I}_{\nsd\times\nsd} & \mathbf{0}_{\nsd \times 1}  \\
     \mathbf{0}_{1 \times \nsd} & 0
    \end{bmatrix}
    ,\qquad
    \ellipticBvec := 
    \begin{bmatrix}
    \adv_{\nsd \times 1} \\ 1
    \end{bmatrix},
\end{align}
where $\mvec{I}$ refers to an identity matrix; and the zero diagonal in $\ellipticAmat$ and the last element in $\ellipticBvec$ respectively refer to the absence of diffusion and the unit advection in the time dimension.
After we finish the formulation with Galerkin/Least squares, we will eventually have a small positive real entry in the last diagonal term along with some non-zero terms in the off-diagonals. Note that we have also assumed an isotropic diffusive medium, so the principal diffusivity values are the same in all directions. This assumption can be trivially relaxed. In the $\adv$ vector, along with the spatial advection components, we have the ``advection in time" component which has the value 1.

This equation can then be rewritten in a more compact form:
\begin{align}\label{pde:generalized-elliptic-operator-form}
    \left(-\gradst\cdot\ellipticAmat\cdot\gradst + \ellipticBvec\cdot\gradst\right)u = f.
\end{align}
\subsection{Space-time inner products and norms}\label{sec:inner-prod-norms}
Let us define the function spaces $V$ and $\spaceVd$ as 
\begin{subequations}
\begin{align}\label{def:function-space-V}
	\spaceV &:= \left\{v\in H^1(\stDom) : v = 0 \textrm{ on }\bit \cup \bsp \right\}, \\
	\spaceVd &:= \left\{ v\in H^1(\stDom) : v = u_0\ \text{on}\ \bit,\ v=g\ \text{on}\ \bsp \right\}.
\end{align}
We also define the following function space of bounded divergence-free convection fields:
\begin{align}
	\spaceVdiv := \setbuilder{v \in \mvec{H}^{1}(\stDom) :\ \divergence v = 0\ \text{in}\ \stDom},
\end{align}
where $ \mvec{H}^{1}(\stDom) = [\mvec{H}^1(\stDom)]^{\nsd} $ is a product function space.
\end{subequations}

Given $v,\ w\in H^1(\stDom)$, we define the $ L^2 $-inner product and the $ L^2 $-norm in $ \stDom $ as
\begin{align*}
(v,w) &:= \int_{\stDom} vw \text{ } d\stDom = \int_{0}^{T}\int_{\spDom} vw \text{ } d\spDom dt, \\
\norm{v} &:= \sqrt{\inner{v}{v}}.
\end{align*}
On any $ \Gamma \subset \partial \stDom $, we define the $ L^2 $-inner product and norm by using a subscript:, e.g.,
\begin{align*}
(v,w)_{\Gamma} := \int_{\Gamma} vw \ d\Gamma \quad \text{and} \quad \norm{u}_{\Gamma} := \sqrt{\inner{u}{u}_{\Gamma}}.
\end{align*}
A special case is the integration over the final time boundary $ \bft $, for which we will use:
\begin{align*}
(v,w)_{\bft} := \int_{\spDom} v(\xvec,T) \ w(\xvec,T) \ d\spDom \quad \text{and} \quad \norm{u}_{\bft} := \sqrt{\inner{u}{u}_{\bft}}.
\end{align*}
Unless otherwise stated, an inner product or a norm without any subscript will be assumed  to be that calculated in $ \stDom $. A typical integration-by-parts in the \textit{spatial dimensions} can be written as:
\begin{align}
\inner{v}{\laplacian w} = \intTimeSpace v\laplacian w \ d\spDom \ dt &= \intTime \left[ -\intSpace \grad v \cdot \grad w\ d\spDom + \int_{S=\partial \spDom} v (\normalvec\cdot\grad w) \ dS \right]dt \\
&= -\inner{\grad v}{\grad w} + \inner{v}{\normalvec\cdot \grad w}_{\bsp}.
\end{align}


Given $ \adv \in \spaceVdiv $ and $ F\in L^2(\stDom) $, the variational problem for \eqref{pde:linear-advection-diffusion} is then to find $u \in \spaceVd$ such that
\begin{equation}\label{var-prob:continuous-variational-form}
	\bform(u,v) = l(v)\quad \forall v \in V,
\end{equation}
where
\begin{subequations}\label{def:continuous-bilinear-linear-forms}
	\begin{align}
		\bform(u,v) &:= (u_t,v) + (\adv\cdot\grad u, v) + (\visco\grad u, \grad v), \label{def:continuous-bilinear-form}\\
		l(v) &:= (F,v). \label{def:continuous-linear-form}
	\end{align}
\end{subequations}
\subsection{Discretization, and discrete inner products and norms}\label{sec:discrete-inner-prod-norm}
We define a tessellation $\mesh$ as the partition of $\stDom$ into a finite number of non-overlapping elements such that $\bar{\stDom} = \bigcup_{K \in \mesh} K$. In the sequel, we will consider a sequence of tessellations $\{\mesh^i\}$ for $i = 0,\ 1,\ \ldots,\ \nrefine$. Here, $\mesh^0$ is the initial mesh, $ \nrefine $ is the number of refined meshes, and for $i \geq  1$, each $\mesh^i$ is obtained by refining some (or all) elements of $\mesh^{i-1}$. In this work, we will restrict ourselves to tessellations where each finite element is a $(\nsd+1)$-dimensional hypercube. And each refined mesh is obtained by performing a ``bisection-type'' mesh refinement, i.e., by bisecting the $d$-dimensional edges of the current hypercube element, and then joining the opposite edges.

Following Bank et al. \cite{bank1983some}, we define a regular vertex as a vertex of $\mesh$ that is a corner of each unrefined element it touches. All other vertices are called irregular. And we call a mesh $k$-irregular if the maximum number of irregular vertices in the mesh is no greater than $k$. In this work, we will restrict ourselves to \emph{1-irregular meshes}, i.e., we refine any unrefined element which contains more than two irregular nodes on any of its edges. This strategy is also known as $2:1$ balancing, and more details on this technique can be found in \cite{bank1983some,sundar2008bottom}.

\begin{remark}
	We assume that each tesselation $\mesh^i$ is \emph{shape regular}, i.e., there exists a constant $\shapeRegularConst$ such that
	\begin{align} \label{eq:shape-regularity-condition}
		\frac{h_K}{\rho_K} \leq \shapeRegularConst \ \forall K \in \mesh^i,
	\end{align}
	for $i = 0,\ 1,\ 2, \ldots,\ \nrefine$, where $h_K$ is the diameter of the element $K$, and $ \rho_K $ is the supremum of the diameters of all the spheres contained in the element $K$.
\end{remark}

The following discussion applies to any given mesh $\mesh^i$, therefore we drop the superscript $i$. Over a given mesh $\mesh$, we define the discrete function spaces $\spaceVh \subset \spaceV$ and $\spaceVhd \subset \spaceVd$ as
\begin{align}\label{def:functiona-space-Vh}
	\spaceVh &:= \left\{\vh \in \spaceV: \vh|_K \in P_k(K), K \in \mesh \right\}, \\
	\spaceVhd &:= \left\{\vh \in \spaceVd: \vh|_K \in P_k(K), K \in \mesh \right\}.
\end{align}
$P_k(K),\ k \geq 1$, is the space of tensor-product polynomial functions in $K$ where the tensor-product basis functions are constructed by polynomials of degree $k$ in each dimension.

We can then define the discrete inner product on this mesh as
\begin{align} \label{def:discrete-norm}
\innerh{u}{v} := \Xsum_{K \in \mesh} \int_{K} uv \ dK,
\end{align}
and the associated discrete norm as
\begin{align}
\normh{\uh} := \sqrt{\innerh{\uh}{\uh}}.
\end{align}

\subsection{Stabilized variational problem} \label{sec:stab-var-prob}
If we naively cast \eqref{var-prob:continuous-variational-form} in terms of the discrete function spaces, the discrete solutions can exhibit numerical instabilities because the bilinear form $\bform(\uh,\vh)$ is not strongly coercive in $\spaceVh$ \cite[see][Sec 10.1.4-10.1.7]{larson2013finite},\cite[Lemma 1]{langer2016space}. This issue can be tackled by adding some amount of numerical diffusion in the time direction or by adding upwind type correction to the weighting functions ~\cite{langer2016space,  steinbach2015space}. 
In what follows, we try to establish a stable variational form of the space-time advection-diffusion problem \eqref{pde:linear-advection-diffusion} by applying the Galerkin/least square (GLS) method \cite{hughes1989new}. The resulting discrete problem is stable and converges to the original PDE in the limit of the mesh size approaching zero.


An equation of the form $\opLT u=F$ can be cast into a least square minimization problem as
\begin{align}\label{eq:least-square-minimization}
    \uh = \argmin_{v\in \spaceVhd} \normh{\opLT v-F} ^2.
\end{align}
The resulting Euler-Lagrange equation corresponding to this minimization problem is
\begin{align}\label{eq:least-square-euler-lagrange}
    \innerh{\opLT \uh}{\opLT \vh} = \innerh{F}{\opLT \vh} \ \forall \vh \in \spaceVh.
\end{align}
In the GLS formulation, equation \eqref{eq:least-square-euler-lagrange} is added to equation \eqref{var-prob:continuous-variational-form} (specifically, its discrete counterpart), i.e., we seek $\uh \in \spaceVhd$ such that,
\begin{align} 
	\label{eq:gls-abstract}
    (\opLT \uh, \vh) + \glspar(\opLT \uh,\opLT \vh)_h &= (F,\vh) + \glspar(F,\opLT \vh)_h, \quad\forall \vh \in V_h    
\end{align}
where $\glspar$ is a positive piecewise-constant function to be chosen later (see \remref{remark:stabilization}).
Thus, given $ \adv \in \spaceVdiv $ and $ F \in L^2(\stDom) $, the discrete variational problem with Galerkin--least-square stabilization reads as: find $\uh \in \spaceVhd$ such that
\begin{align}\label{def:discrete-variational-form}
    \bformh(\uh,\vh) = \lformh(\vh) \quad\forall \vh \in V_h,
\end{align}
where the discrete bilinear form $\bformh(\uh,\vh)$ and linear form $\lformh(\vh)$ are as follows:
\begin{align}\label{def:discrete-bilinear-form}
\begin{split}
    \bformh(\uh,\vh) &:= (\uht,\vh) + (\adiv\uh, \vh) + (\visco\grad\uh,\grad\vh)\\
    &\quad +\innerh{\glspar [\uht + \adiv \uh- \visco\laplacian \uh]}{[\vht + \adiv \vh- \visco\laplacian\vh]},
\end{split}\\
\label{def:discrete-linear-form}
\lformh(\vh) &:= \inner{F}{\vh} + \innerh{F}{[\vht + \adiv \vh- \visco\laplacian\vh]}. 
\end{align}

\begin{remark} \label{remark:gls-consistency}
	If the exact solution $ u $ of \eqref{pde:linear-advection-diffusion} belongs to $ H^1(\temporalInterval; H^2(\spDom))$, then the Galerkin--Least-square formulation \eqref{eq:gls-abstract} is consistent, i.e., the exact solution $u$ satisfies \eqref{eq:gls-abstract}. We can rewrite \eqref{eq:gls-abstract} as
	\begin{align}
	\innerh{\opLT \uh - F}{\vh + \glspar\opLT \vh} &= 0.
	\end{align}
	Now, if we substitute $ \uh = u $ in the left hand side of the above equation, we have
	\begin{align}\label{eq:gls-abstract-for-exact-solution}
	\innerh{\opLT u - F}{\vh + \glspar\opLT \vh} &= 0 \ \ \forall \vh \in \spaceVh
	\end{align}
	because $ \opLT u - F = 0 $ pointwise (see \eqref{pde:linear-advection-diffusion-equation}). It follows that 
	\begin{align}\label{eq:exact-solution-discrete-form}
	\bformh(u,\vh) = \lformh(\vh) \quad\forall \vh \in \spaceVh
	\end{align}
\end{remark}
\begin{remark}\label{remark:galerkin-orthogonality}
	Under the same assumptions as \remref{remark:gls-consistency}, subtracting \eqref{eq:exact-solution-discrete-form} from  \eqref{def:discrete-variational-form}, we get the so-called Galerkin orthogonality:
	\begin{align*}
	\bformh(\uh-u,\vh) = 0 \quad\forall \vh \in \spaceVh.
	\end{align*}
\end{remark}
\begin{remark} \label{remark:stabilization}
	The expression for the stabilization function $ \glspar $ for the space-time problem \eqref{def:discrete-variational-form} is adapted from existing estimates \cite{codina2000stabilized}, and is given for element $K$ as
	\begin{align} \label{eq:glspar-formula}
		\glspar_K = \left[ \frac{c_1}{\he^2} \visco + \frac{c_2}{\he} \norm{\advst} \right]^{-1},
	\end{align}
	where $\advst$ is the space-time advection defined in \secref{subsec:pde}, and $ c_1,c_2 >0 $ are some positive constants. An alternative definition of $ \glspar $ can be found in \cite{shakib1991new}.
\end{remark}


\subsection{Analysis of the stabilized formulation}
We now turn to the analysis of stability and convergence of the method discussed in the previous section. We will prove the stability and the convergence estimates in the following discrete norm:
\begin{align}
\normVh{\uh} := \left[\norm{\uh}_{\bft}^2 + \visco\norm{\grad\uh}^2 + \normh{\glspar^{\halfnice}[\uht + \adiv\uh - \visco\laplacian\uh]}^2\right]^{\halfnice}
\label{eq:discrete-norm}
\end{align}
We also define an auxiliary norm as follows:
\begin{align}
\normVhstar{\uh}^2 := \left(\norm{\uht}^2 + \normVh{\uh}^2\right)^{\halfnice}.
\end{align}
Note that $\normVh{\uh} \leq \normVhstar{\uh}$ by definition. We will use both these norms in the sequel to prove the stability and boundedness of the discrete bilinear form \eqref{def:discrete-bilinear-form}.

\begin{lemma}[Boundedness]\label{lemma:boundedness}
	For $ \adv \in \spaceVdiv $ and $ \uh \in \spaceVhd,\ \vh \in \spaceVh $, the bilinear form \eqref{def:discrete-bilinear-form} is uniformly bounded on $\spaceVhd \times \spaceVh$, i.e.,
	\begin{align}
		|\bformh(\uh,\vh)|\leq\mu_b \normVh{\uh} \normVhstar{\vh}.
	\end{align}
\end{lemma}
\begin{proof}
	See \ref{sec:proof-boundedness}.
\end{proof}

\begin{lemma}[Coercivity]\label{lemma:coercivity-stability}
If $ \adv \in \spaceVdiv $ and $ \uh \in \spaceVh $, then
\begin{align*}
    \bformh(\uh,\uh) \geq \mu_c\normVh{\uh}^2,\ \text{where}\ \mu_c =\halfnice.
\end{align*}
\end{lemma}
\begin{proof}
	See \ref{sec:proof-coercivity}.
\end{proof}

\lemref{lemma:coercivity-stability} imply that \eqref{def:discrete-variational-form} has a unique solution.

\begin{corollary}
	Given $ \adv \in \spaceVdiv $, there exists a unique solution $\uh$ of \eqref{def:discrete-variational-form}, such that $\normVh{\uh} \leq \nicefrac{1}{\mu_c} \normVhstar{F} $.
\end{corollary}
\begin{proof}
	From \eqref{def:discrete-linear-form}, we have
	\begin{align*}
	|\lformh(\vh)| &\leq |\inner{F}{\vh}| + |\innerh{F}{\glspar[\vht + \adiv \vh- \visco\laplacian\vh]}| \\
	&\leq \norm{F}\norm{\vh} + \normh{\glspar^{\halfnice} F} \normh{\glspar^{\halfnice} \opLT \vh} \\
	&\leq \left[ \norm{F}^2 + \normh{\glspar^{\halfnice} F}^2 \right]^{\halfnice} \left[ \norm{\vh}^2 + \normh{\glspar^{\halfnice} \opLT \vh}^2 \right]^{\halfnice}.
	\end{align*}
	Defining
	\begin{align*}
	 \normVhstar{F} = \left[ \norm{F}^2 + \normh{\glspar^{\halfnice} F}^2 \right]^{\halfnice},
	\end{align*}
	and using $ \norm{\vh} \leq \norm{\grad \vh} $, we have
	\begin{align} \label{eq:lh-bound}
	|\lformh(\vh)| &\leq \normVhstar{F} \normVh{\vh}.
	\end{align}
	This proves that the linear form $ \lformh $ is continuous. Since by \lemref{lemma:boundedness} and \lemref{lemma:coercivity-stability}, we also have that the discrete bilinear form $ \bformh $ is both bounded and coercive, therefore by the Lax-Milgram Lemma \cite{larson2013finite, jt1976introduction}, \eqref{def:discrete-variational-form} has a unique solution in $ \spaceVh $. Moreover, using \lemref{lemma:coercivity-stability} \eqref{def:discrete-variational-form}, and \eqref{eq:lh-bound}, we have
	\begin{align}
		\mu_c \normVh{\uh}^2 \leq \bformh(\uh, \uh) = \lformh(\uh) \leq \normVhstar{F} \normVh{\uh},\ \ \text{i.e.,}\ \mu_c \normVh{\uh} \leq \normVhstar{F}.
	\end{align}
\end{proof}

\subsubsection{A priori error analysis}
In this section, we derive \textit{a priori} error estimates for the GLS stabilized space-time advection diffusion equation. In \eqref{pde:linear-advection-diffusion}, we assume that the exact solution $ u \in \spaceV \cap H^{s}(\stDom) $ (using regularity estimates from Section 6.3 in \cite{evans2022partial}), and consequently the forcing  $ F \in L^2(\stDom) \cap H^{s-2}(\stDom) $. To establish the error estimates, we consider an interpolant $ \interp{u} : \spaceV \cap H^{s}(\stDom) \rightarrow \spaceVh $ associated with the finite element mesh $\mesh$, such that $\interp{u}$ is nodally exact, i.e., $\interp{u}(\xvec_i) = u(\xvec_i)$ where $\xvec_i$ is a regular nodal point in $\mesh$. Then we have the following estimates for the distance between $ u $ and $ \interp{u} $.
\begin{lemma}[Approximation estimate in a Sobolev space]\label{estimate:sobolev-approximation} Assume $\mesh$ be a given mesh where \eqref{eq:shape-regularity-condition} holds.
Let $\interpolant$ be the interpolation operator from $ \spaceV \cap H^{s}(\stDom) $ to $\spaceVh$, with order of interpolation $k \in \mathbb{N}$. Also, let $ l \in \mathbb{N} $. If $k, l, s $ satisfy $0\leq l \leq (k+1) \leq s $, then
\begin{align*}
    \normH[l]{v-\interp{v}}{K}\leq \approxConstant h^{(k+1-l)}\seminormH[k+1]{v}{K}\ \forall K \in \mesh,
\end{align*}
where the constant $\approxConstant$ only depends on $k,l,s$ and is independent of $h$ and $v$.
\end{lemma}
\begin{proof}
The proof follows from the technique outlined in Section 4.4 of \cite{brenner2007mathematical} and Section 6.7 of \cite{jt1976introduction}.
\end{proof}
\begin{remark}
The following estimates are direct consequences of \lemref{estimate:sobolev-approximation} on any element $ K \in \mesh $ (assuming $s \geq k+1$).
\begin{subequations}\label{estimate:sobolev-interp}
	\begin{align}
	\normL{v-\interp{v}}{K} &\leq {\approxConstant}_0 \he^{k+1}\seminormH[k+1]{v}{K},
	\label{estimate:sobolev-interp-L2}\\
	\normL{\grad (v-\interp{v})}{K} &\leq {\approxConstant}_1 \he^{k}\seminormH[k+1]{v}{K},
	\label{estimate:sobolev-interp-grad-L2}\\
	\normL{\partial_t(v-\interp{v})}{K} &\leq {\approxConstant}_1 \he^{k}\seminormH[k+1]{v}{K},
	\label{estimate:sobolev-interp-time-der-L2} \\
	\normL{\laplacian (v-\interp{v})}{K} &\leq {\approxConstant}_2 \he^{k-1}\seminormH[k+1]{v}{K}.
	\label{estimate:sobolev-interp-laplacian-L2}
	\end{align}
\end{subequations}
\begin{corollary} \label{lemma:interp-boundary-estimate}
	For the final time boundary $ \bft $, we similarly have
	\begin{align}
	\normL{v-\interp{v}}{\partial K} 
	&\leq C_{a\Gamma} \he^{k+\frac{1}{2}}|v|_{H^{k+1}(K)}.
	\label{estimate:sobolev-interp-boundary-L2}
	\end{align}
\end{corollary}
\begin{proof}
	See \ref{sec:proof-interp-boundary}.
\end{proof}
\end{remark}
\begin{lemma}\label{lemma:interp-estimate-in-Vh}
Let $ u \in \spaceV \cap H^{s}(\stDom) $, and suppose $ \interpolant: \spaceV \cap H^{s}(\stDom) \rightarrow \spaceVh  $ is the projection of $u$ from $\spaceV$ to $\spaceVh$. Assume that $ \interpolant $ has the degree of interpolation $ k \in \mathbb{N} $ such that $ k+1 \leq s $. If $ s \geq 2 $, then the following estimate holds,
\begin{align*}
    \normVh{\interp{u}-u} \leq C h^{k}\seminormH[k+1]{u}{\stDom}
\end{align*}
for some $C>0$.
\end{lemma}
\begin{proof}
	See \ref{proof:interp-estimate-vh}.
\end{proof}

\begin{theorem}\label{theorem:a-priori-estimate}
Assume that $u \in V\cap H^{s}(\stDom)$ is the exact solution to \eqref{pde:linear-advection-diffusion} and $\uh\in V_h$ is the solution to the finite element problem \eqref{def:discrete-variational-form}. If the degree of the basis functions in $ \spaceVh $ is $ k $, and the element-wise stabilization function is given by \eqref{eq:glspar-formula}, then the discretization error estimate satisfies,
\begin{align*}
    \normVh{\uh-u}\leq C h^{k}\normH[k+1]{u}{\stDom}
\end{align*}
for some $C>0$.
\end{theorem}
\begin{proof}
	See \ref{sec:proof-a-priori}.
\end{proof}


\subsubsection{A posteriori error analysis}
In addition to solving \eqref{def:discrete-bilinear-form} in space-time, we would also like to perform adaptive refinement of the space-time mesh $ \mesh $. To this end, we use the following residual-based \textit{a posteriori} error indicator for each element $ K $ in $ \mesh $:
\begin{align}
\etaElm = \left[ h_K^2\|r\|_K^2 + \frac{1}{2}\Xsum_{E\in\mathcal{E}_K}h_E\|j\|_E^2 \right]^{\halfnice},
\label{eq:eta-element}
\end{align}
where $ \he $ is the size of the element $ K $, and $ \mathcal{E}_K $ is the set of all the boundaries of element $ K $. The PDE-residual $ r $ and the jump-residual $ j $ are defined as
\begin{subequations}
	\begin{align}
	r|_K &= F - \opLT \uh \ \ \text{on every} \ K \in \mesh, \\
	j|_E &= \begin{cases}
	0\ \ \text{if} \ E \in \bit \cup \bsp,\\
	-\jump(\visco\normalvec_E\cdot\grad\uh),\ \ \text{otherwise}.
	\end{cases}
	\end{align}
\end{subequations}
The jump operator $ \jump $ acting on a function $ \vh $ is defined as
\begin{align}
\jump_E \vh(\xvec) = \lim_{t\rightarrow 0^+} \vh(\xvec - t\normalvec_E) - \lim_{t\rightarrow 0^+} \vh(\xvec + t\normalvec_E) ,\ E \in \edgeset,
\end{align}
$\edgeset$ being the set of all the edges of all the elements in the mesh $\mesh$. The cumulative error indicator $ \eta $ can be defined as
\begin{align}
\etaFull := \left(\elmsum \etaElm^2 \right)^{\halfnice}.
\end{align}
The next theorem shows that the cumulative error indicator bounds the true error. A similar error estimate has been used for the heat equation in \cite{steinbach20197}.

\begin{theorem}\label{theorem:a-posteriori-estimate}
If $ u $ is the solution of \eqref{pde:linear-advection-diffusion} and $ \uh \in \spaceVh $ is the solution of \eqref{def:discrete-variational-form}, then 
	\begin{align}
	\normVhstar{u-\uh} 
	&\leq C\etaFull.
	\end{align}
\end{theorem}
\begin{proof}
	See \ref{sec:proof-a-posteriori}.
\end{proof}

\section{Implementation details}
\label{sec:implementation}
In this section, we briefly discuss the implementation details. We have used the open-source FEM library \textsc{Deal.II} \cite{arndt2022deal} and our in-house program to obtain all the solutions presented in \secref{sec:numerical-examples}. We utilize the parallel implementation of finite element analysis using the standard continuous Galerkin method with globally $C^0$ Lagrangian basis functions \cite{hughes2003finite}. The underlying mesh is based on a tree-type data-structure \cite{sundar2008bottom, ishii2019solving}, where each ``leaf'' node represents a (hypercube shaped) ``finite element.'' Local refinements are achieved by bisecting the faces of the hypercube elements, and joining opposite faces. These local refinements sometimes introduce ``hanging nodes'' \cite{bank1983some, ainsworth1997aspects}. In our implementation, we treat such a node as a constrained degree of freedom, i.e., the value of the solution at a hanging node is forced to be the average of the corner nodes on the same face \cite[Sec 3.2]{ainsworth1997aspects}, \cite[Sec 5]{fries2011hanging}.

We use \petsc{} \cite{balay2019petsc} for the solution of the linear algebra problems after the finite element discretization is performed. All the problems in \secref{sec:numerical-examples} are solved using the biconjugate gradient squared (BCGS) \cite{saad2003iterative} solver from \petsc{} in combination with a preconditioner based on the additive schwarz method (ASM) \cite{saad2003iterative, brenner2007mathematical}.

\section{Numerical Examples}\label{sec:numerical-examples}
In this section, we look at three specific examples modeled by \eqref{pde:linear-advection-diffusion}. These special cases are: \textit{(i)} the heat equation ($\adv = \underline{0}\text{ , }\visco \neq 0$), \textit{(ii)} the advection-diffusion equation (both $\adv \neq \underline{0} \text{ , }\visco \neq 0$) and \textit{(iii)} the transport equation ($\adv \neq \underline{0} \text{ , } \visco=0$, which is a hyperbolic equation). 

In the first two cases, we perform a convergence study with a known smooth analytical solution and confirm the results of a priori and a posteriori errors obtained in \secref{sec:math-formulation}. Then we discuss the nature of the uniform mesh solutions for the second and third cases and present comparisons with a sequential time-marching solution. Finally, we present results on the space-time adaptive solutions for all three cases.
In all examples, we take the space-time domain to be $\stDom = [0,1]^3 $.
  
\subsection{Convergence study}\label{subsec:conv-study}
\begin{figure}[!htb]
\centering
\begin{minipage}{.49\textwidth}
  \centering
  \begin{tikzpicture}
      \begin{loglogaxis}[
          width=0.99\linewidth, 
          xlabel=$h$, 
          legend style={at={(0.95,0.05)},anchor=south east,legend columns=1}, 
          x tick label style={rotate=0,anchor=north}, 
          xtick={0.015625, 0.03125, 0.0625, 0.125},
          xticklabels={$2^{-6}$,$2^{-5}$,$2^{-4}$,$2^{-3}$},
        ]
        \addplot table[x expr={\thisrow{h}},y expr={\thisrow{eta}},col sep=comma]{results_data/convergence-results-new/heat/errors_case_3_nu_1e-2_p_1.txt};
        \addplot table[x expr={\thisrow{h}},y expr={\thisrow{err_h}},col sep=comma]{results_data/convergence-results-new/heat/errors_case_3_nu_1e-2_p_1.txt};
        \addplot table[x expr={\thisrow{h}},y expr={\thisrow{err_l2}},col sep=comma]{results_data/convergence-results-new/heat/errors_case_3_nu_1e-2_p_1.txt};
        \logLogSlopeTriangle{0.3}{0.1}{0.65}{1}{red};
        \logLogSlopeTriangle{0.3}{0.1}{0.43}{1}{blue};
        \logLogSlopeTriangle{0.3}{0.1}{0.12}{2}{brown};
        \legend{\small $\etaFull$, $\normVh{\uh-u}$, $\normL[2]{\uh-u}{\stDom}$}
      \end{loglogaxis}
    \end{tikzpicture}
    \subcaption{$\adv = (0,0),\ \visco=10^{-2}$}
    \label{fig:conv-study-heat-expdecay-1e-2}
\end{minipage}
\begin{minipage}{.49\textwidth}
	\centering
	\begin{tikzpicture}
		\begin{loglogaxis}[
			width=0.99\linewidth, 
			xlabel=$h$, 
			legend style={at={(0.95,0.05)},anchor=south east,legend columns=1}, 
			x tick label style={rotate=0,anchor=north}, 
			xtick={0.015625, 0.03125, 0.0625, 0.125},
			xticklabels={$2^{-6}$,$2^{-5}$,$2^{-4}$,$2^{-3}$},
			]
			\addplot table[x expr={\thisrow{h}},y expr={\thisrow{eta}},col sep=comma]{results_data/convergence-results-new/advdiff/errors_case_3_nu_1e-2_p_1.txt};
			\addplot table[x expr={\thisrow{h}},y expr={\thisrow{err_h}},col sep=comma]{results_data/convergence-results-new/advdiff/errors_case_3_nu_1e-2_p_1.txt};
			\addplot table[x expr={\thisrow{h}},y expr={\thisrow{err_l2}},col sep=comma]{results_data/convergence-results-new/advdiff/errors_case_3_nu_1e-2_p_1.txt};
			\logLogSlopeTriangle{0.3}{0.1}{0.42}{1.3}{blue};
			\logLogSlopeTriangle{0.3}{0.1}{0.59}{1}{red};
			\logLogSlopeTriangle{0.3}{0.1}{0.12}{1.5}{brown};
			\legend{\small $\etaFull$, $\normVh{\uh-u}$, $\normL[2]{\uh-u}{\stDom}$}
		\end{loglogaxis}
	\end{tikzpicture}
	\subcaption{$\adv = 2\pi (-y+\halfnice,\ x-\halfnice),\ \visco=10^{-2}$}
	\label{fig:conv-study-advdiff-expdecay-1e-2}
\end{minipage}
\\~\\~\\
\begin{minipage}{.49\textwidth}
	\centering
	\begin{tikzpicture}
		\begin{loglogaxis}[
			width=0.99\linewidth, 
			xlabel=$h$, 
			legend style={at={(0.95,0.05)},anchor=south east,legend columns=1}, 
			x tick label style={rotate=0,anchor=north}, 
			xtick={0.015625, 0.03125, 0.0625, 0.125},
			xticklabels={$2^{-6}$,$2^{-5}$,$2^{-4}$,$2^{-3}$},
			]
			\addplot table[x expr={\thisrow{h}},y expr={\thisrow{eta}},col sep=comma]{results_data/convergence-results-new/heat/errors_case_3_nu_1e-6_p_1.txt};
			\addplot table[x expr={\thisrow{h}},y expr={\thisrow{err_h}},col sep=comma]{results_data/convergence-results-new/heat/errors_case_3_nu_1e-6_p_1.txt};
			\addplot table[x expr={\thisrow{h}},y expr={\thisrow{err_l2}},col sep=comma]{results_data/convergence-results-new/heat/errors_case_3_nu_1e-6_p_1.txt};
			\logLogSlopeTriangle{0.3}{0.1}{0.5}{2}{red};
			\logLogSlopeTriangle{0.3}{0.1}{0.12}{2}{blue};
			\legend{\small $\etaFull$, $\normVh{\uh-u}$, $\normL[2]{\uh-u}{\stDom}$}
		\end{loglogaxis}
	\end{tikzpicture}
	\subcaption{$\adv = (0,0),\ \visco=10^{-6}$}
	\label{fig:conv-study-heat-expdecay-1e-6}
\end{minipage}
\begin{minipage}{.49\textwidth}
	\centering
	\begin{tikzpicture}
		\begin{loglogaxis}[
			width=0.99\linewidth, 
			xlabel=$h$, 
			legend style={at={(0.95,0.05)},anchor=south east,legend columns=1}, 
			x tick label style={rotate=0,anchor=north}, 
			xtick={0.015625, 0.03125, 0.0625, 0.125},
			xticklabels={$2^{-6}$,$2^{-5}$,$2^{-4}$,$2^{-3}$},
			]
			\addplot table[x expr={\thisrow{h}},y expr={\thisrow{eta}},col sep=comma]{results_data/convergence-results-new/advdiff/errors_case_3_nu_1e-6_p_1.txt};
			\addplot table[x expr={\thisrow{h}},y expr={\thisrow{err_h}},col sep=comma]{results_data/convergence-results-new/advdiff/errors_case_3_nu_1e-6_p_1.txt};
			\addplot table[x expr={\thisrow{h}},y expr={\thisrow{err_l2}},col sep=comma]{results_data/convergence-results-new/advdiff/errors_case_3_nu_1e-6_p_1.txt};
			\logLogSlopeTriangle{0.3}{0.1}{0.3}{2}{blue};
			\logLogSlopeTriangle{0.3}{0.1}{0.48}{1.5}{red};
			\logLogSlopeTriangle{0.3}{0.1}{0.12}{2}{brown};
			\legend{\small $\etaFull$, $\normVh{\uh-u}$, $\normL[2]{\uh-u}{\stDom}$}
		\end{loglogaxis}
	\end{tikzpicture}
	\subcaption{$\adv = 2\pi (-y+\halfnice,\ x-\halfnice),\ \visco=10^{-6}$}
	\label{fig:conv-study-advdiff-expdecay-1e-6}
\end{minipage}
\caption{Convergence of the error in the $L^2$-norm, the error in the $\normVh{\cdot}$-norm, and the error indicator $\etaFull$ for a sequence of uniformly refined 3D space-time meshes. Left column plots are for the heat equation ($\adv=\mathbf{0}$), and right column for advection-diffusion equation ($\adv\neq\mathbf{0}$). Additional results on the convergence of the advection-diffusion equation can be found in \ref{sec:conv-study-higher-order}.}
\label{fig:conv-study-uniref}
\end{figure}
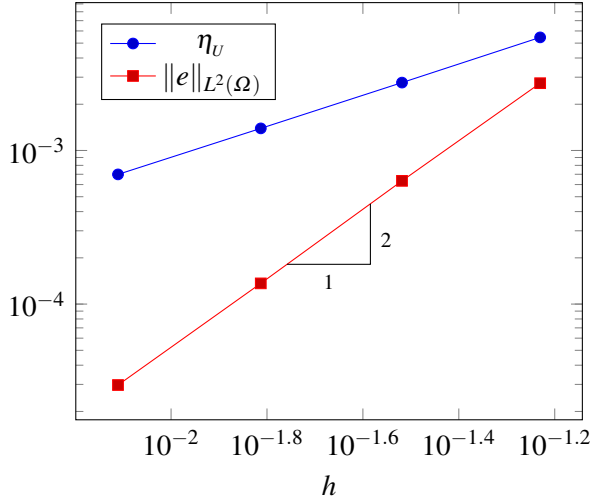
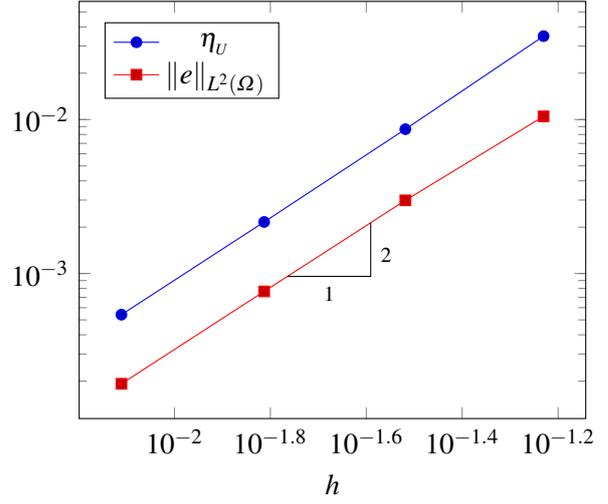

\subsubsection{The heat equation in 2D ($ \spDom \subset \mathbb{R}^2,\ \stDom \subset \mathbb{R}^3 $)}
\label{subsubsec:convergence-heat-eq}
The linear heat equation ($ d=2 $) is given by 
\begin{align}
    \begin{cases}
        u_t - \grad \cdot(\visco\grad u) = f \qquad &\text{in } \stDom = (0,1)\times(0,1)\times(0,1]\\
        u = 0 & \text{on } \Gamma_s\\
        u = u_0(\mvec{x}) & \text{on } \Gamma_0.\\
    \end{cases}
    \label{eq:numex-pde-heat-eq}
\end{align}
This equation is obtained by simply setting $ \adv = (0,0) $ in \eqref{pde:linear-advection-diffusion-equation}. For the convergence studies, we choose $\visco = 10^{-2}$, and the forcing $f$ is obtained by assuming the solution
\begin{align}
    u_a = e^{-t}\sin{(2\pi x)}\sin{(2\pi y}),
\end{align}
with initial condition $u_0(x,y) = u_a(x,y,t=0) = \sin{(2\pi x)}\sin{(2\pi y)}$. 

This problem is then solved through the formulation presented in \secref{sec:math-formulation} on different sizes of space-time mesh. 
\figref{fig:conv-study-heat-expdecay-1e-2} and \ref{fig:conv-study-heat-expdecay-1e-6} show the plot of $\normL{u^h-u_a}{\stDom}$, $ \normVh{u^h-u_a} $ and $\etaFull$ against $h$ (on a $\log-\log$ plot) for $\visco = 10^{-2}$ and $10^{-6}$ respectively. The rate of decrease in $ \normVh{u^h-u_a} $ with respect to $ h $ can be compared to the result in \thmref{theorem:a-priori-estimate} with $ k=1 $. 

We select these two values of diffusivities to illustrate the behavior of the error estimates and the actual errors in the low- and high- diffusivity limits. Across all diffusivities, the slope of $\normL[2]{\uh-u}{\stDom}$ is 2. In the low diffusivity range (see \figref{fig:conv-study-heat-expdecay-1e-6}) the slope of the discrete norm is 2, while in the high diffusivity range it drops to 1 (see \figref{fig:conv-study-heat-expdecay-1e-2}). This is because in the high diffusivity range, the $\visco \norm{\grad \uh}^2$ term in \eqref{eq:discrete-norm} dominates. Similarly, the slope of the estimator $\etaFull$ drops from 2 to 1 as we move from a low diffusivity range to the high diffusivity range. This is explained by looking at the two terms of the expression in \eqref{eq:eta-element}. The jump terms on the element edges (that are multiplied by $h$) become important in the high-diffusivity range. 

\subsubsection{The advection-diffusion equation in 2D ($ \spDom \subset \mathbb{R}^2,\ \stDom \subset \mathbb{R}^3 $)}
\label{subsubsec:convergence-adv-diff-eq}
The advection-diffusion equation ($ \nsd = 2 $) is given by (once again $\stDom$ is a unit cube),
\begin{align}
    \begin{cases}
        u_t + \adv\cdot\grad u- \grad \cdot(\visco\grad u) = f \qquad &\text{in } \stDom = (0,1)\times(0,1)\times(0,1]\\
        u = 0 & \text{on } \Gamma_s\\
        u = u_0(\mvec{x}) & \text{on } \Gamma_0\\
    \end{cases}
    \label{eq:numex-adv-diff-pde}
\end{align}
where we set $ \adv(\mvec{x}) = (-2\pi [y-\halfnice],\ 2\pi [x-\halfnice]) $, and we repeat the same process as in \secref{subsubsec:convergence-heat-eq} for $\visco = 10^{-2}$ and $10^{-6}$. The corresponding convergence results are shown in \figref{fig:conv-study-advdiff-expdecay-1e-2} and \ref{fig:conv-study-advdiff-expdecay-1e-6}. Once again, we see that all three quantities perform better as $\visco$ is decreased. Additional results (higher order basis functions, across various $ \visco $ values) can be found in \ref{sec:conv-study-higher-order}.




\begin{figure}
    \centering
    \begin{minipage}[t]{.32\textwidth}
    \centering
    \begin{tikzpicture}[scale=0.3]
    \draw (0,0) rectangle (7,7);
    \draw[ultra thick] (2,2) circle (0.8cm);
    \fill[blue!40!white] (2,2) circle (0.8cm);
    \draw[red,thick,dashed](-0.3,4.3) -- (4.2,-0.2);
    \draw[thick,->] (0.2,2.7) -- (0.75,3.25);
    \draw[thick,->] (2.7,0.2) -- (3.25,0.75);
    \draw[thick,->] (2,2) arc (-135:170:2cm);
    \node[above right=2pt of {(0,4)}, outer sep=0.5pt] {A};
    \node[above right=2pt of {(4,0)}, outer sep=0.5pt] {B};
    \node[above left=3pt of {(0,3)}, outer sep=0.5pt] {$y$};
    \node[below right=3pt of {(3,0)}, outer sep=0.5pt] {$x$};
    \node[below=0pt of {(0,0)}, outer sep=0pt] {$(0,0)$};
    \node[below=0pt of {(7,0)}, outer sep=0pt] {$(1,0)$};
    \node[above=0pt of {(0,7)}, outer sep=0pt] {$(0,1)$};
    \node[below=0pt of {(7,0)}, outer sep=0pt] {$(1,0)$};
    \node[above=0pt of {(7,7)}, outer sep=0pt] {$(1,1)$};
    \end{tikzpicture}
	\subcaption{Location of the pulse in the $ (x,y) $ plane}
	\label{fig:initial-pulse-location}
    \end{minipage}
    \begin{minipage}[t]{.32\textwidth}
    \centering
    \begin{tikzpicture}[scale=1]
    \begin{axis}[
          width=0.7\linewidth, 
          xlabel=$s$, 
          ylabel=$u$,
          legend cell align={left},
          legend style={at={(1,1)},anchor=north east,legend columns=1}, 
          x tick label style={rotate=0,anchor=north} 
        ]
        \addplot+ [color=blue,thick,mark=none]
        table[x expr={\thisrow{"x"}},y expr={\thisrow{"u"}},col sep=comma]{results_data/case_2_ad_smoothic_uniref/case2_initial_condition_cross_section.txt};
      \end{axis}
    \node (x) at (0.3,0.5) {A};
    \node (y) at (1.9,0.5) {B};
    \end{tikzpicture}
	\subcaption{Smooth pulse at $ t = 0 $ \\ (\secref{subsubsec:case2-smoothic})}
	\label{fig:ad-case2-initial-condition}
    \end{minipage}
	\begin{minipage}[t]{.32\textwidth}
		\centering
		\begin{tikzpicture}[scale=1]
			\begin{axis}[
				width=0.7\linewidth, 
				xlabel=$s$, 
				ylabel=$u$,
				legend cell align={left},
				legend style={at={(1,1)},anchor=north east,legend columns=1}, 
				x tick label style={rotate=0,anchor=north} 
				]
				\addplot+ [color=blue,thick,mark=none]
				table[x expr={\thisrow{"arc_length"}},y expr={\thisrow{"u_initial"}},col sep=comma]{results_data/ts_case3_ad_sharpic_uniref/k_0_mesh_128/case3_mesh128_plot_line_data.txt};
			\end{axis}
			\node (x) at (0.3,0.5) {A};
			\node (y) at (1.9,0.5) {B};
		\end{tikzpicture}
		\subcaption{Discontinous pulse at $ t = 0 $ \\ (\secref{subsubsec:case3-sharpic})}
		\label{fig:ad-case3-initial-condition}
	\end{minipage}
    \caption{Initial condition for the advection-diffusion equation examples described in \secref{subsubsec:case2-smoothic} (Gaussian pulse) and \secref{subsubsec:case3-sharpic} (discontinuous pulse): (a) the circular arrow shows the advection field in the domain and the colored circle shows the location of the initial pulse, (b) a line cut of the Gaussian pulse through the slant line AB, (c) a line cut of the discontinuous pulse through the slant line AB}
    \label{fig:ad-case2-case3-initial-condition}
\end{figure}
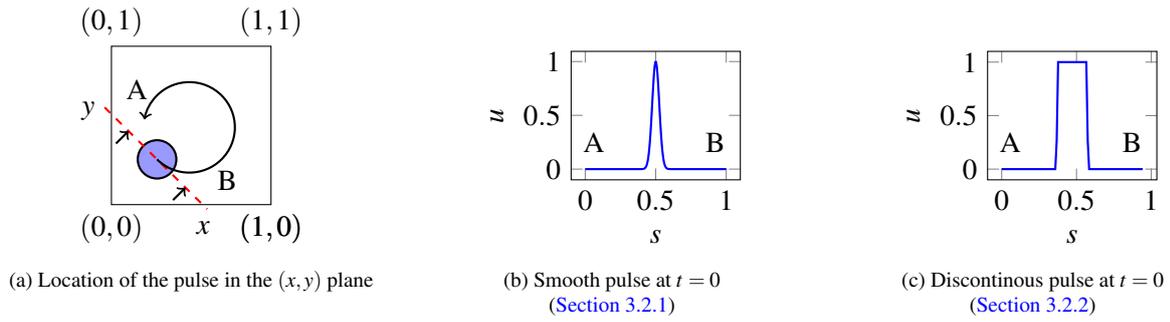
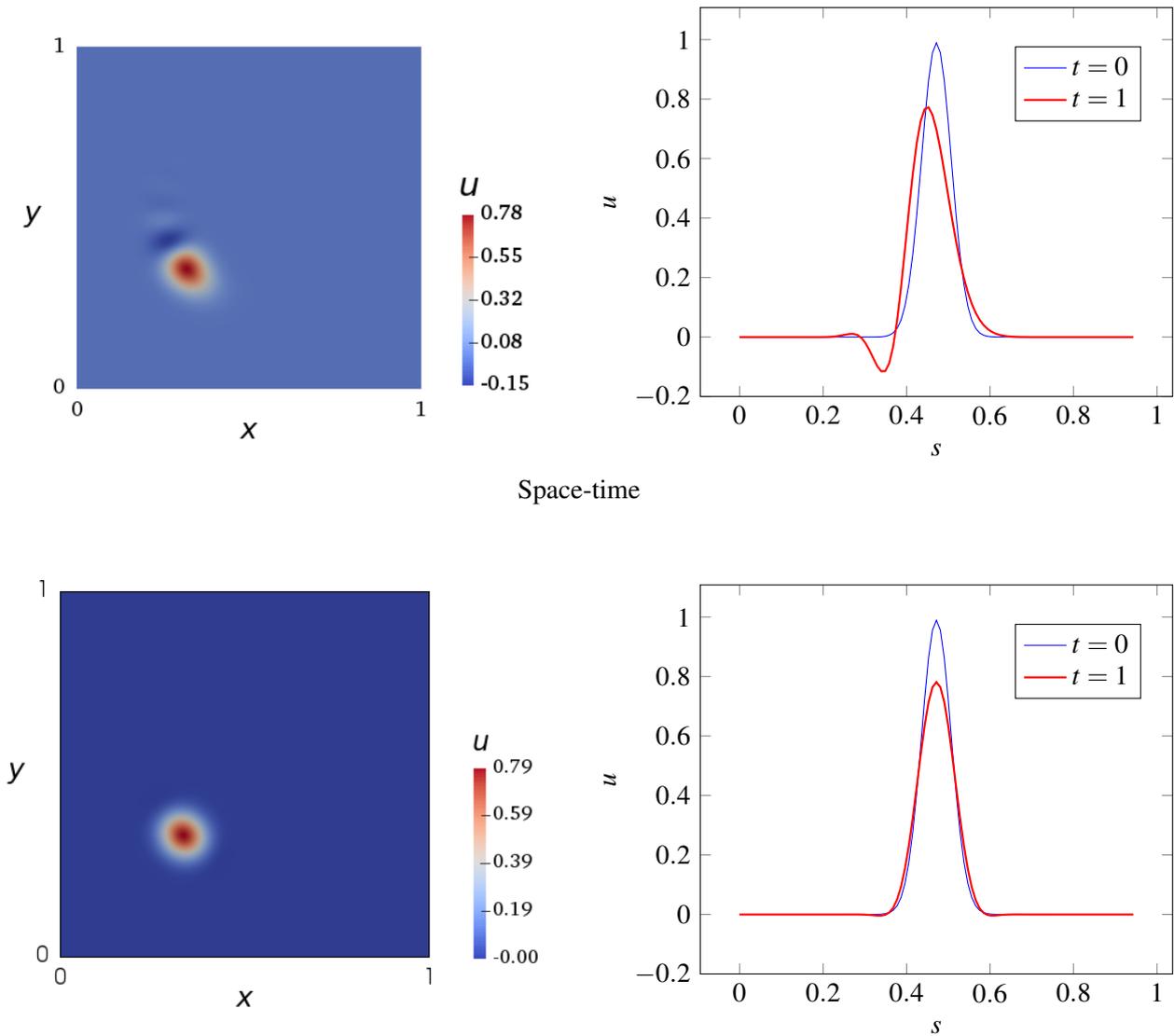
\begin{figure}[]
    \centering
    Sequential (Crank-Nicolson)\\
    \begin{minipage}[t]{.49\textwidth}
      \centering
        \includegraphics[width=0.99\linewidth]{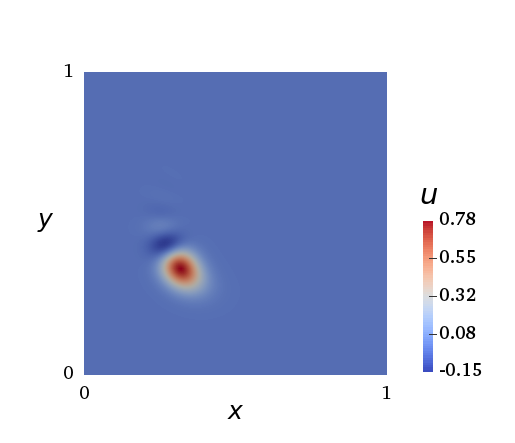}
    \end{minipage}
    \hskip 5pt
    \begin{minipage}[t]{.49\textwidth}
      \centering
      \begin{tikzpicture}
      \begin{axis}[
          width=0.99\linewidth, 
          xlabel=$s$, 
          ylabel=$u$,
          ymin=-0.2,
          legend style={at={(0.8,0.9)},anchor=north,legend columns=1}, 
          x tick label style={rotate=0,anchor=north} 
        ]
        \addplot+ [color=blue,mark=none]
        table[x expr={\thisrow{"arc_length"}},y expr={\thisrow{"u_initial"}},col sep=comma]{results_data/ts_case2_ad_smoothic_uniref/k_1e-4_mesh_128/case2_mesh128_plot_line_data.txt};
        \addplot+ [color=red,thick, mark=none]
        table[x expr={\thisrow{"arc_length"}},y expr={\thisrow{"u_curr"}},col sep=comma]{results_data/ts_case2_ad_smoothic_uniref/k_1e-4_mesh_128/case2_mesh128_plot_line_data.txt};
        \legend{ $t=0$, $t=1$}
      \end{axis}
    \end{tikzpicture}
    \end{minipage}\\
    Space-time\\
    \begin{minipage}[t]{.49\textwidth}
      \centering
      \includegraphics[width=0.99\linewidth]{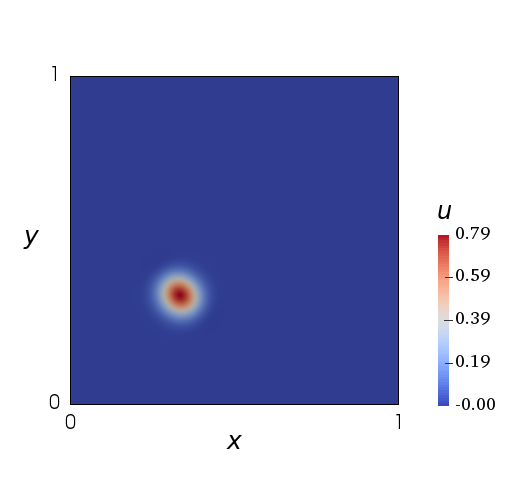}
    \end{minipage}
    \hskip 5pt
    \begin{minipage}[t]{.49\textwidth}
      \centering
      \begin{tikzpicture}
      \begin{axis}[
          width=0.99\linewidth, 
          xlabel=$s$, 
          ylabel=$u$,
          ymin=-0.2,
          legend style={at={(0.8,0.9)},anchor=north,legend columns=1}, 
          x tick label style={rotate=0,anchor=north} 
        ]
        \addplot+ [color=blue,mark=none]
        table[x expr={\thisrow{"arc_length"}},y expr={\thisrow{"u"}},col sep=comma]{results_data/case_2_ad_smoothic_uniref/k_1e-4_uniref_7/case2_k_1e-4_uniref_7_profile_t_0.txt};
        \addplot+ [color=red,thick, mark=none]
        table[x expr={\thisrow{"arc_length"}},y expr={\thisrow{"u"}},col sep=comma]{results_data/case_2_ad_smoothic_uniref/k_1e-4_uniref_7/case2_k_1e-4_uniref_7_profile_t_1.txt};
        \legend{ $t=0$, $t=1$}
      \end{axis}
    \end{tikzpicture}
    \end{minipage}
    \caption{Comparison between sequential time stepping (top row) and space-time solution (bottom row) on mesh size $128^3$ (see \secref{subsubsec:case2-smoothic}). An initial pulse ($t=0$) is rotated under an advection field. The pulse at $t=1$ is plotted.}
    \label{fig:case2-mesh-128-uniref-lev7-comparison}
\end{figure}
\begin{figure}[]
    \centering
    \begin{minipage}[t]{.49\textwidth}
      \centering
      \includegraphics[width=0.99\linewidth]{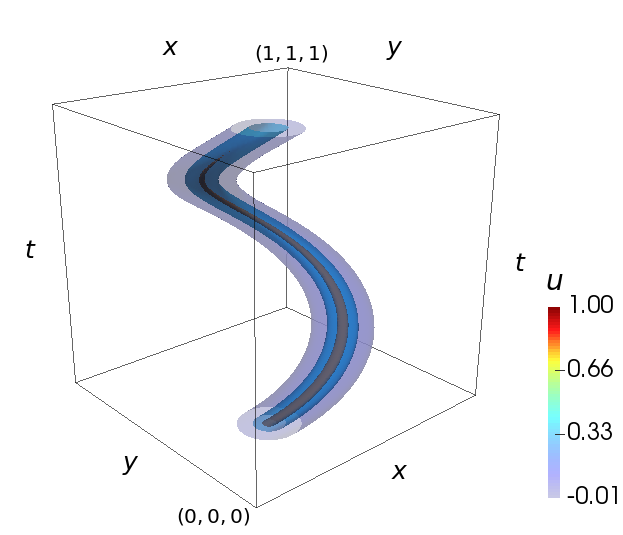}
    \end{minipage}
    \caption{A contour plot of the rotating pulse in the space-time solution with $128^3$ grid (see \secref{subsubsec:case2-smoothic}). The rotating pulse renders a helical shape in the space-time domain}
    \label{fig:case2-ad-uniref-7-threshold-view}
\end{figure}

\subsection{Uniform space-time mesh solutions and comparison to time-marching solutions}
\subsubsection{Advection-diffusion with smooth initial condition}\label{subsubsec:case2-smoothic}
We consider the two-dimensional, linear advection-diffusion equation \eqref{eq:numex-adv-diff-pde} where the advection field $\adv$ is purely rotational with unit angular velocity (see \figref{fig:initial-pulse-location}, \ref{fig:ad-case2-initial-condition}) and is given by 
\begin{align}
    \adv(\mvec{x}) = (-2\pi r_y, 2\pi r_x),
    \label{eq:numex-adv-field-rotating}
\end{align}
where $r_x$ and $r_y$ denote the distance of any point in the spatial domain from the center of the spatial domain $(\frac{1}{2},\frac{1}{2})$, i.e., $r_x=(x-0.5), r_y=(y-0.5)$. The speed of rotation is chosen in such a way that the pulse completes a full revolution at $t=1$.

The initial condition $u_0:\spDom \rightarrow \mathbb{R}$ is a smooth function and is given by
\begin{align}
    u_0(\mvec{x}) = 
    e^{-\frac{(x-a)^2+(y-b)^2}{d^2}}.
    \label{eq:numex-case2-initial-condition}
\end{align}
This is essentially a Gaussian pulse with its center at $(a,b)=\left(\frac{1}{3},\frac{1}{3}\right)$ and thickness at base $\approx 2d = 0.01$. The force $f$ on the right hand side is zero. The diffusivity value $\visco$ is fixed at $10^{-4}$. The initial pulse keeps rotating in the domain as time evolves. Since the advection field has unit angular velocity, thus theoretically the center of the pulse at $t=0$ and $t=1$ should coincide. Since the field is also diffusive, the height of the pulse reduces with time.

\figref{fig:case2-ad-uniref-7-threshold-view} shows contours of the solution $ u $ in space-time. For clarity, only some of the iso-surfaces close to the pulse are shown. The rotation of the pulse is evident from a helical structure of the figure. The total number of elements in this case is $128\times 128\times 128$ with each element being a trilinear Lagrange element.

To see how the space-time solution behaves in comparison to the time-marching methods, we choose the Crank-Nicolson scheme to solve the same problem. This is reasonable since Crank-Nicolson is the simplest second order implicit time-marching method. To compare the methods, we discretize $\stDom=(0,1)^2\times(0,1]$ with the same number of elements in each dimension. As an example, \figref{fig:case2-mesh-128-uniref-lev7-comparison} compares the solution contours obtained by the sequential method on a 2D $(128\times 128)$ mesh marching over $128$ time steps against a space-time solution in 3D using a $128^3$ mesh. The right column of plots in this figure shows a cross-section of the pulses at both $t=0$ and $t=1$ along the plane `AB', which passes through the center of initial pulse and is tangential to the local velocity vector (see \figref{fig:ad-case2-initial-condition}).

The Crank-Nicolson method is dispersive in nature and the solution exhibits a phase error; therefore the centers of the pulses between the two methods do not match. On the other hand, the space-time solution shows little dispersion and thus the peak centers align exactly. In regard to the undershoot around the pulse, the Crank-Nicolson solution shows a phase lag, whereas the space-time solution is visibly symmetric about the centre of the pulse. In both the solutions, however, the height of the peak at $t=1$ is roughly the same.

\subsubsection{Pure advection with a discontinuous initial condition}\label{subsubsec:case3-sharpic}

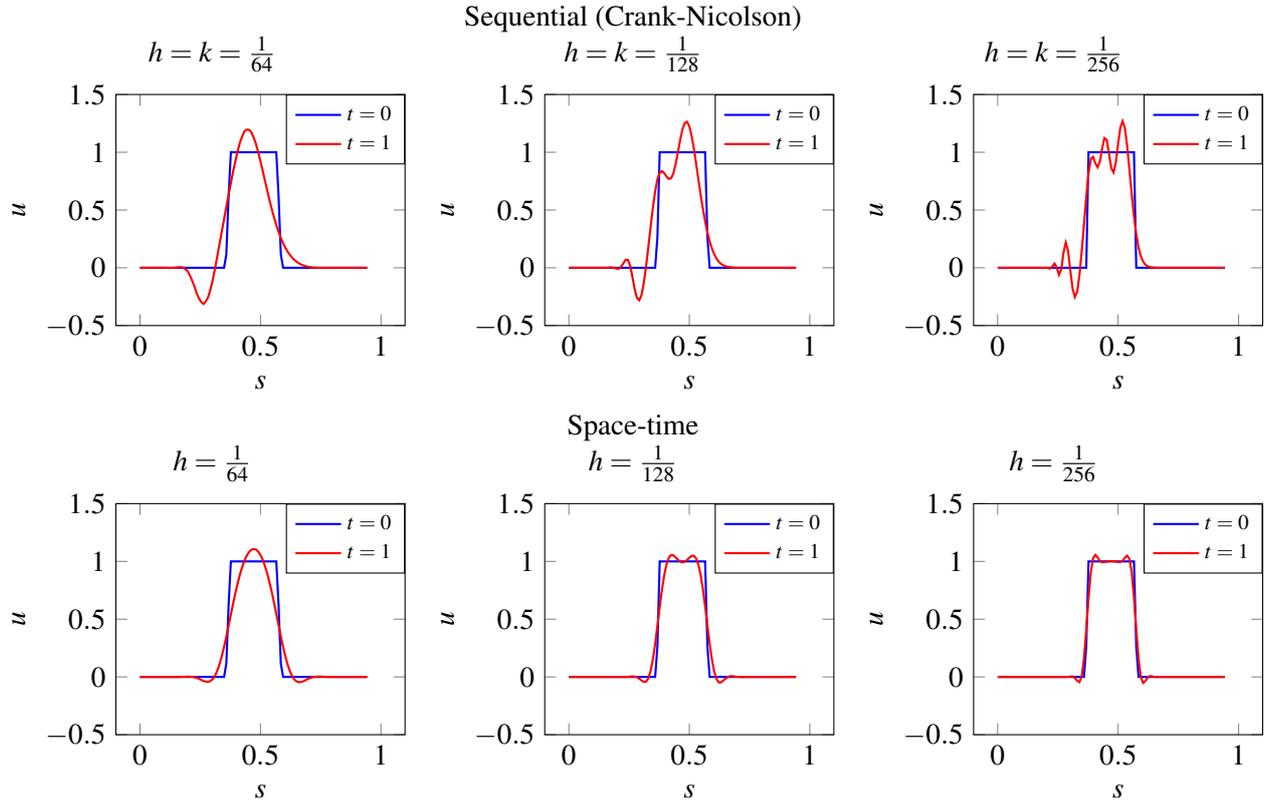
\begin{figure}[!htb]
    \centering
    Sequential (Crank-Nicolson)\\
    \begin{minipage}[t]{.32\textwidth}
      \centering
      $h = k = \frac{1}{64}$
    \end{minipage}
    \begin{minipage}[t]{.32\textwidth}
      \centering
      $h = k = \frac{1}{128}$
    \end{minipage}
    \begin{minipage}[t]{.32\textwidth}
      \centering
      $h = k = \frac{1}{256}$
    \end{minipage}\\
    \begin{minipage}[t]{.32\textwidth}
      \centering
      \begin{tikzpicture}
      \begin{axis}[
          width=0.99\linewidth, 
          xmin=-0.1,
          xmax=1.1,
          ymin=-0.5,
          ymax=1.5,
          xlabel=$s$, 
          ylabel=$u$,
          legend style={at={(1,1)},anchor=north east,legend columns=1}, 
          x tick label style={rotate=0,anchor=north} 
        ]
        \addplot+ [color=blue,thick,mark=none]
        table[x expr={\thisrow{"arc_length"}},y expr={\thisrow{"u_initial"}},col sep=comma]{results_data/ts_case3_ad_sharpic_uniref/k_0_mesh_064/case3_mesh64_plot_line_data.txt};
        \addplot+ [color=red,thick,mark=none]
        table[x expr={\thisrow{"arc_length"}},y expr={\thisrow{"u_curr"}},col sep=comma]{results_data/ts_case3_ad_sharpic_uniref/k_0_mesh_064/case3_mesh64_plot_line_data.txt};
        \legend{\scriptsize {$t=0$}, \scriptsize{$t=1$}}
      \end{axis}
    \end{tikzpicture}
    \end{minipage}
    \hspace{0.1mm}
    \begin{minipage}[t]{.32\textwidth}
      \centering
      \begin{tikzpicture}
      \begin{axis}[
          width=0.99\linewidth, 
          xmin=-0.1,
          xmax=1.1,
          ymin=-0.5,
          ymax=1.5,
          xlabel=$s$, 
          ylabel=$u$,
          legend style={at={(1,1)},anchor=north east,legend columns=1}, 
          x tick label style={rotate=0,anchor=north} 
        ]
        \addplot+ [color=blue,thick,mark=none]
        table[x expr={\thisrow{"arc_length"}},y expr={\thisrow{"u_initial"}},col sep=comma]{results_data/ts_case3_ad_sharpic_uniref/k_0_mesh_128/case3_mesh128_plot_line_data.txt};
        \addplot+ [color=red,thick, mark=none]
        table[x expr={\thisrow{"arc_length"}},y expr={\thisrow{"u_curr"}},col sep=comma]{results_data/ts_case3_ad_sharpic_uniref/k_0_mesh_128/case3_mesh128_plot_line_data.txt};
        \legend{\scriptsize {$t=0$}, \scriptsize{$t=1$}}
      \end{axis}
    \end{tikzpicture}
    \end{minipage}
    \hspace{0.1mm}
    \begin{minipage}[t]{.32\textwidth}
      \centering
      \begin{tikzpicture}
      \begin{axis}[
          width=0.99\linewidth, 
          xmin=-0.1,
          xmax=1.1,
          ymin=-0.5,
          ymax=1.5,
          xlabel=$s$, 
          ylabel=$u$,
          legend style={at={(1,1)},anchor=north east,legend columns=1}, 
          x tick label style={rotate=0,anchor=north} 
        ]
        \addplot+ [color=blue,thick,mark=none]
        table[x expr={\thisrow{"arc_length"}},y expr={\thisrow{"u_initial"}},col sep=comma]{results_data/ts_case3_ad_sharpic_uniref/k_0_mesh_256/case3_mesh256_plot_line_data.txt};
        \addplot+ [color=red,thick,mark=none]
        table[x expr={\thisrow{"arc_length"}},y expr={\thisrow{"u_curr"}},col sep=comma]{results_data/ts_case3_ad_sharpic_uniref/k_0_mesh_256/case3_mesh256_plot_line_data.txt};
        \legend{\scriptsize {$t=0$}, \scriptsize{$t=1$}}
      \end{axis}
    \end{tikzpicture}
    \end{minipage}\\
    Space-time\\
    \begin{minipage}[t]{.32\textwidth}
      \centering
      $h = \frac{1}{64}$
    \end{minipage}
    \begin{minipage}[t]{.32\textwidth}
      \centering
      $h = \frac{1}{128}$
    \end{minipage}
    \begin{minipage}[t]{.32\textwidth}
      \centering
      $h = \frac{1}{256}$
    \end{minipage}\\
    \begin{minipage}[t]{.32\textwidth}
      \centering
      \begin{tikzpicture}
      \begin{axis}[
          width=0.99\linewidth, 
          xmin=-0.1,
          xmax=1.1,
          ymin=-0.5,
          ymax=1.5,
          xlabel=$s$, 
          ylabel=$u$,
          legend style={at={(1,1)},anchor=north east,legend columns=1}, 
          x tick label style={rotate=0,anchor=north} 
        ]
        \addplot+ [color=blue,thick,mark=none]
        table[x expr={\thisrow{"arc_length"}},y expr={\thisrow{"u"}},col sep=comma]{results_data/case_3_ad_sharpic_uniref/k_0_uniref_6/case3_uniref_6_profile_t_0.txt};
        \addplot+ [color=red,thick,mark=none]
        table[x expr={\thisrow{"arc_length"}},y expr={\thisrow{"u"}},col sep=comma]{results_data/case_3_ad_sharpic_uniref/k_0_uniref_6/case3_uniref_6_profile_t_1.txt};
        \legend{\scriptsize {$t=0$}, \scriptsize{$t=1$}}
      \end{axis}
    \end{tikzpicture}
    \end{minipage}
    \hspace{0.1mm}
    \begin{minipage}[t]{.32\textwidth}
      \centering
      \begin{tikzpicture}
      \begin{axis}[
          width=0.99\linewidth, 
          xmin=-0.1,
          xmax=1.1,
          ymin=-0.5,
          ymax=1.5,
          xlabel=$s$, 
          ylabel=$u$,
          legend style={at={(1,1)},anchor=north east,legend columns=1}, 
          x tick label style={rotate=0,anchor=north} 
        ]
        \addplot+ [color=blue,thick,mark=none]
        table[x expr={\thisrow{"arc_length"}},y expr={\thisrow{"u"}},col sep=comma]{results_data/case_3_ad_sharpic_uniref/k_0_uniref_7/case3_uniref_7_profile_t_0.txt};
        \addplot+ [color=red,thick,mark=none]
        table[x expr={\thisrow{"arc_length"}},y expr={\thisrow{"u"}},col sep=comma]{results_data/case_3_ad_sharpic_uniref/k_0_uniref_7/case3_uniref_7_profile_t_1.txt};
        \legend{\scriptsize {$t=0$}, \scriptsize{$t=1$}}
      \end{axis}
    \end{tikzpicture}
    \end{minipage}
    \hspace{0.1mm}
    \begin{minipage}[t]{.32\textwidth}
      \centering
      \begin{tikzpicture}
      \begin{axis}[
          width=0.99\linewidth, 
          xmin=-0.1,
          xmax=1.1,
          ymin=-0.5,
          ymax=1.5,
          xlabel=$s$, 
          ylabel=$u$,
          legend style={at={(1,1)},anchor=north east,legend columns=1}, 
          x tick label style={rotate=0,anchor=north} 
        ]
        \addplot+ [color=blue,thick,mark=none]
        table[x expr={\thisrow{"arc_length"}},y expr={\thisrow{"u"}},col sep=comma]{results_data/case_3_ad_sharpic_uniref/k_0_uniref_8/case3_uniref_7_profile_t_0.txt};
        \addplot+ [color=red,thick,mark=none]
        table[x expr={\thisrow{arc_length}},y expr={\thisrow{u}},col sep=comma]{results_data/case_3_ad_sharpic_uniref/k_0_uniref_8/case3_uniref_8_profile_t_1.txt};
        \legend{\scriptsize {$t=0$}, \scriptsize{$t=1$}}
      \end{axis}
    \end{tikzpicture}
    \end{minipage}
    \caption{Sequential and space-time solutions for the case described in Section \ref{subsubsec:case3-sharpic} on a uniform mesh of size $128^3$.}
    \label{fig:case3-mesh-128-uniref-lev7-comparison}
\end{figure}

We solve \eqref{eq:numex-adv-diff-pde} with the same advection field mentioned in \eqref{eq:numex-adv-field-rotating} but this time with an initial condition that is not smooth in space. The initial data is given by,
\begin{align}
    u_0(\mvec{x}) = \begin{cases}
    1 & \text{if } (r_1^2 + r_2^2) \leq 1,\\
    0 & \text{otherwise},
    \end{cases}
    \label{eq:numex-case3-initial-condition}
\end{align}
where $\mvec{r} = (\mvec{x}-\mvec{x_0})/\sigma$. Here $\mvec{x_0} = \left(\frac{1}{3},\frac{1}{3}\right)$ is the initial position of the center of the circular pulse. The radius of this pulse is 1. Clearly, $u_0$ is discontinuous in space, (see \figref{fig:initial-pulse-location}, \ref{fig:ad-case3-initial-condition}).
The diffusivity value $\visco$ is set to $10^{-8}$ in this case. As earlier, the source term $f$ is zero. This choice of a negligible value of $\visco$ effectively renders this case as a purely advective one. The global Peclet number is given as $ Pe_g = \frac{\modulus{a}\cdot L}{\visco} \approx 4.4 \times 10^8 $. As in the previous case, we discretize this problem through space-time as well as sequential time marching schemes and compare the pulse at $t=0$ and $t=1$. \figref{fig:case3-mesh-128-uniref-lev7-comparison} shows this comparison for both sequential and space-time method for three sizes of discretizations: $ h=\oneOver{64}, \ \oneOver{128} $ and $  \oneOver{256} $. The corresponding mesh-Peclet numbers are $ Pe = \frac{\modulus{a}\cdot h}{2\visco} \approx 3.47\times 10^6, \ 1.74\times 10^6 $ and $ 8.68 \times 10^5 $ respectively. As in the previous example, the plots are line cuts of the pulse onto the plane `AB'.


It can be noticed immediately that both sequential and space-time formulation have difficulty approximating the discontinuous pulse, which can be attributed to the use of continuous Galerkin approximation when attempting to model a solution that is discontinuous. The final time representation of the pulse gets smoothened out in both cases, albeit to a different degree. But once again, the Crank-Nicolson method exhibits higher dispersion and phase errors, whereas the space-time solution has zero phase error and a smaller dispersion. Once again, as in the previous example, the undershoot in the Crank-Nicolson method only takes place in the upwind direction whereas there is typically no undershoot in the downwind direction. But the space-time solution does not display any directional preference for the undershoot.

\subsection{Adaptive solutions}\label{subsec:adaptive-solutions}
The examples considered in the previous section have solutions that show a high degree of spatial as well as temporal localization. That is, at a given instance of time, the solution function has a significant change in value only at some small area of the whole domain and zero at all other points of the domain. Moreover, as time evolves, there is a limited region where the solution changes. A significant portion of the domain never experiences any change in the solution with the evolution of time. This can be observed in \figref{fig:case2-ad-uniref-7-threshold-view}, where the solution is zero in all of the white region. This kind of problems are therefore perfect candidates where space-time adaptive refinement strategy can be useful. To show the effectiveness of adaptive refinement in a space-time simulation, we consider the following three examples. 

\subsubsection{The heat equation: estimator behavior in adaptive refinement}
\begin{figure}[!htb]
\centering
\begin{minipage}{.49\textwidth}
  \begin{center}
  \begin{tikzpicture}
      \begin{loglogaxis}[
          width=0.99\linewidth, 
          xlabel=$\textsc{\# DOF}$, 
          legend style={at={(0.95,0.95)},anchor=north east,legend columns=1}, 
          x tick label style={rotate=0,anchor=north}
        ]
        \addplot+
        table[x expr={(\thisrow{nodes})},y expr={\thisrow{L2_Err}},col sep=space]{results_data/case_7_diff_gaussian_adref/k_1e-2_etol_1e-4/case7_error_data_uniref.txt};
        \addplot+
        table[x expr={(\thisrow{nodes})},y expr={\thisrow{L2_Err}},col sep=space]{results_data/case_7_diff_gaussian_adref/k_1e-2_etol_1e-4/case7_error_data_adref.txt};
        \logLogSlopeTriangleFlipped{0.8}{0.1}{0.4}{0.66}{$ \nicefrac{2}{3} $}{blue};
        \logLogSlopeTriangleDecreasing{0.48}{0.1}{0.2}{0.96}{$ \approx 1 $}{red};
        \legend{\small{$\|e\|_{L^2}$, Uniform},\small{$\|e\|_{L^2}$, Adaptive}}
      \end{loglogaxis}
    \end{tikzpicture}
    \subcaption{Convergence of $\normL{e}{\Omega}$}
    \label{fig:case5-diff-conv-error-compare-adref-uniref}
    \end{center}
\end{minipage}
\hskip 5pt
\begin{minipage}{.49\textwidth}
  \begin{center}
  \begin{tikzpicture}
      \begin{loglogaxis}[
          width=0.99\linewidth, 
          xlabel=$\textsc{\# DOF}$, 
          legend style={at={(0.95,0.95)},anchor=north east,legend columns=1}, 
          x tick label style={rotate=0,anchor=north}
        ]
        \addplot+
        table[x expr={(\thisrow{nodes})},y expr={\thisrow{eta}},col sep=space]{results_data/case_7_diff_gaussian_adref/k_1e-2_etol_1e-4/case7_error_data_uniref.txt};
        \addplot+
        table[x expr={(\thisrow{nodes})},y expr={\thisrow{eta}},col sep=space]{results_data/case_7_diff_gaussian_adref/k_1e-2_etol_1e-4/case7_error_data_adref.txt};
        \logLogSlopeTriangleFlipped{0.75}{0.1}{0.55}{0.34}{$ \nicefrac{1}{3} $}{blue};
        \logLogSlopeTriangleDecreasing{0.48}{0.1}{0.2}{0.54}{$ \approx \halfnice $}{red};
        \legend{\small{$\etaFull$, Uniform},\small{$\etaFull$, Adaptive}}
        \end{loglogaxis}
    \end{tikzpicture}
    \subcaption{Convergence of $ \etaFull $}
    \label{fig:case5-diff-conv-eta-compare-adref-uniref}
    \end{center}
\end{minipage}
\caption{Comparison of the error indicator and the error in $L^2$-norm in both uniform and adaptive refinement. For the uniform refinement curves (in blue), the slopes of the $ L^2 $-error and the estimator $ \etaFull $ are $\nicefrac{2}{3}$ and $\nicefrac{1}{3}$ respectively. When compared with \figref{fig:conv-study-heat-expdecay-1e-2}, the additional factor of $ \nicefrac{1}{3} $ in the slopes appears here because the $ x $-axis is the number of degrees of freedom (\#DOF)  which is equal to $ h^{-3} $, or $ h = (\#DOF)^{-\nicefrac{1}{3}} $.}
\label{fig:case5-diff-conv-compare-adref-uniref}
\end{figure}
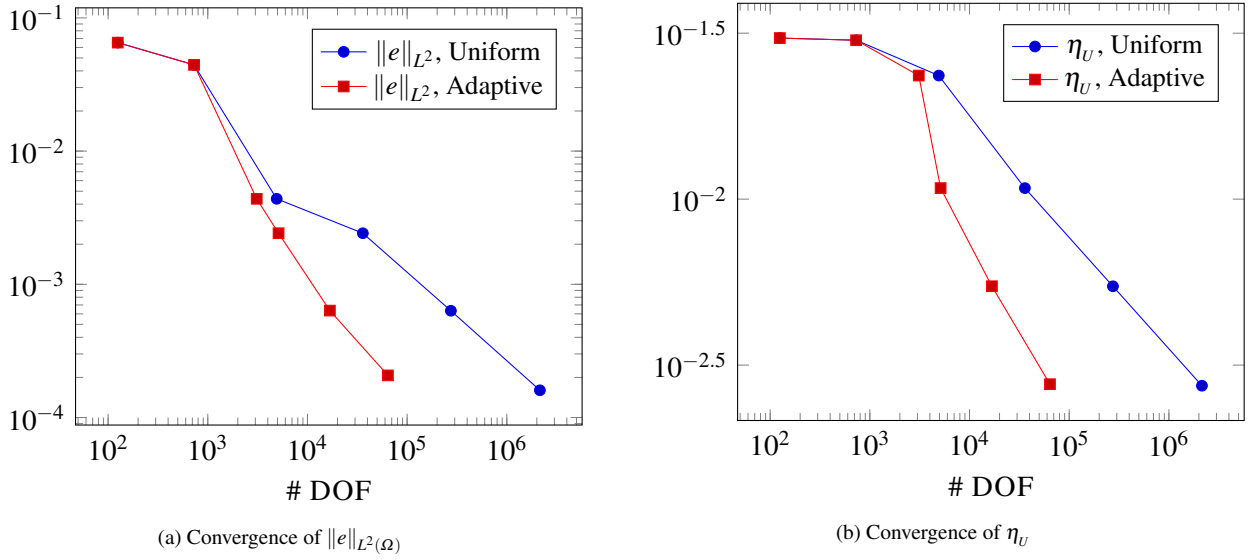

\begin{figure}[!htb]
\centering
\begin{minipage}[t]{.49\textwidth}
\includegraphics[width=0.99\linewidth]{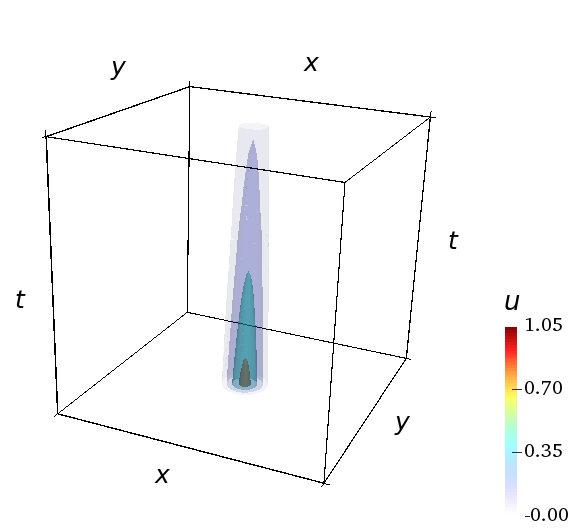}
\end{minipage}
\begin{minipage}[t]{.49\textwidth}
\includegraphics[width=0.99\linewidth]{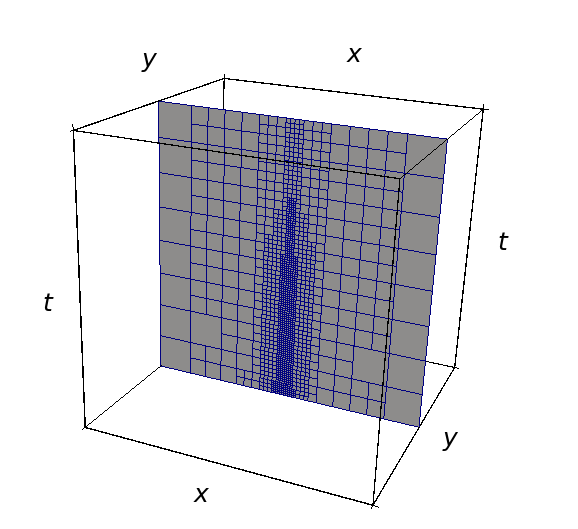}
\end{minipage}
\caption{(Left) contours of the solution to the Gaussian heat source problem \eqref{eq:numex-pde-heat-eq}, (right) a $y-t$ slice at the final state of the adaptively refined mesh}
\label{fig:case5-refined-mesh-slice-yt-plane}
\end{figure}


We begin with a simple heat diffusion problem to illustrate adaptive refinement in space-time. We go back to the heat \eqref{eq:numex-pde-heat-eq} with a forcing $f$ given by
\begin{align}
	f = - \left[ \frac{2}{\theta} + \visco \left(\frac{4}{d^4}\right)  \left( r_x^2 + r_y^2 - d^2 \right) \right] e^{-\nicefrac{2t}{\theta}-\nicefrac{(r_x^2+r_y^2)}{d^2}}.
    \label{eq:case7-gaussian-forcing}
\end{align}
The analytical solution is given by
\begin{align}
    u = e^{-\nicefrac{2t}{\theta} - \nicefrac{(r_x^2+r_y^2)}{d^2}}.
    \label{eq:case7-gaussian-u-analytical}
\end{align}
This problem describes a Gaussian pulse at the center of the spatial domain $U$ diffusing as time evolves.
In both \eqref{eq:case7-gaussian-forcing} and \eqref{eq:case7-gaussian-u-analytical}, $r_x = x-x_0$ and $r_y = y-y_0$, where $(x_0,y_0) = (\nicefrac{1}{2},\nicefrac{1}{2})$ is the center of the heat pulse, $d = 0.05$ is the ``thickness" of the pulse. The other parameter $\theta$ determines the time-scale, and we set it to 1. The diffusivity $\visco = 0.01$ .

To solve this problem adaptively, we begin with a very coarse 3D octree mesh (representing 2D in space and 1D in time). After computing the FEM space-time solution on this coarse mesh, we use the \textit{a posteriori} error estimate presented in \thmref{theorem:a-posteriori-estimate} to calculate the error indicator $\etaElm$ in each element ($K$ $\in$ $\mesh$). The elements, where the indicator $\etaElm$ is larger than a predetermined error tolerance, are refined. The refined mesh is then used to solve the same problem once again. This process is repeated for a few times till the error value in all the elements are smaller then the tolerance. \figref{fig:case5-diff-conv-compare-adref-uniref} shows comparisons of $\normL{u-u^h}{\stDom}$ and $ \etaFull $ for both uniform refinement and adaptive refinement. As expected, the error decreases much more rapidly when adaptive refinement is used.

A slice of the space-time mesh at the final refined state is shown in \figref{fig:case5-refined-mesh-slice-yt-plane}. The slice is a $y$-constant plane, therefore the levels of refinement in the time direction are visible in this image. The largest element size in this slice is $h_{max}=\frac{1}{2^3}=\frac{1}{8}$ whereas the smallest element size is $h_{min}=\frac{1}{2^7}=\frac{1}{128}$. The ratio of $h_{max}$ to $h_{min}$ is thus $16$.

\subsubsection{Advection-diffusion with smooth initial condition}\label{subsubsec:adref-case2-smoothic}
\input{sections/6_06_case2_ad_smoothic.tex}

\figref{fig:case2-adref-final-state-mesh-with-threshold} shows the pulse contour for the advection diffusion problem (solved earlier using a uniform space-time mesh) with a smooth initial condition, along with the adaptively refined space-time mesh. A 2D slice of the same mesh is shown in \figref{fig:case2-adref-final-state-mesh-slice-with-threshold}. \figref{fig:case2-adref-final-state-xy-slice} shows the mesh slice at $t=1$. As expected, the refinement is clustered around the smooth pulse. \figref{fig:case2-adref-final-state-yt-slice} on the other hand, shows a $y-t$ slice of the whole mesh. Once again, as seen from both these slices, the mesh size varies greatly between the finer and the coarser regions, which in this case, differs by a factor of $2^5$, i.e., the largest element has a size that is 32 times the size of the smallest element. Especially in relation to the mesh slice in \figref{fig:case2-adref-final-state-yt-slice}, it can be interpreted that different regions in the spatial domain are subjected to different ``time steps" to reach the same final time. This is a very different behavior compared to a sequential solution where every region in the spatial domain has to go through the same number of time steps to reach the final time. Moreover, in sequential methods, the size of this time-step is often constrained by the minimum size of the mesh. For example, in this problem, the error in the solution is higher near the peak of the pulse and is gradually less as distance from the peak increases in space-time. This indicates that to achieve a reasonable accuracy, the regions near the pulse need to be resolved at least by elements of size $\approx \frac{1}{2^8}$ in space. This then implies that when using implicit time-marching methods such as the Crank-Nicolson method, the time-step size also needs to be $\mathcal{O}(h)$. 

\figref{fig:case2-ad-adref-different-kappas} plots the cross-section of the pulse at $t=0$ and $t=1$ solved through both uniform-refinement and adaptive-refinement for four different values of $\visco$. For all of these four cases, the uniformly refined mesh is of size $128^3$, i.e., $\textsc{DOF} = 2146689$. The $\textsc{DOF}$ in the adaptive meshes vary from case to case. As an example, for the case of $\visco=10^{-8}$, the final refined state has $\textsc{DOF}=192507$.

Note that there is an appreciable decrease in the height of the pulse when $\visco = 10^{-4}$, but there is no such apparent loss when $\visco=10^{-6}$ or $\visco=10^{-8}$ which is correctly captured by the adaptive solutions. But the uniform $128^{3}$ mesh is unable to capture this; indicating that a further finer mesh is required. This shows that when $\visco$ is very low, it is very difficult to resolve the peak without resorting to extremely small mesh sizes. A final observation is that $\visco < 10^{-6}$ can be used to model a ``pure advection" case when considering $t\in[0,1]$ for this problem. Thus, this value of $\visco$ is used to model a ``purely advective" transport problem in the next example. 

\subsubsection{The advection-diffusion with sharp initial condition}\label{subsubsec:adref-case3-sharpic}
\begin{figure}
\centering
\begin{minipage}[t]{.48\textwidth}
\includegraphics[width=0.99\linewidth]{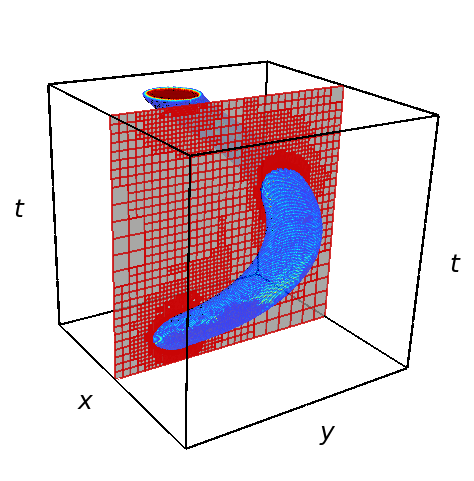}
\caption{Contour of the discontinuous pulse through the space-time mesh}
 \label{fig:case3-adref-final-state-contour-view}
\end{minipage}
\hspace{0.5mm}
\begin{minipage}[t]{.48\textwidth}
  \centering
    \begin{tikzpicture}
      \begin{axis}[
          width=0.99\linewidth, 
          xlabel=$s$, 
          ylabel=$u$,
          legend style={at={(0.8,0.9)},anchor=north,legend columns=1}, 
          x tick label style={rotate=0,anchor=north} 
        ]
        \addplot+ [color=blue,mark=none]
        table[x expr={\thisrow{arc_length}},y expr={\thisrow{u}},col sep=comma]{results_data/case_3_ad_sharpic_adref/adref_k_1e-8/case3_adref_k_1e-8_at_t_0.txt};
        \addplot+ [color=red,thick, mark=none]
        table[x expr={\thisrow{arc_length}},y expr={\thisrow{u}},col sep=comma]{results_data/case_3_ad_sharpic_adref/adref_k_1e-8/case3_adref_k_1e-8_at_t_1.txt};
        \legend{ $t=0$, $t=1$}
      \end{axis}
    \end{tikzpicture}
    \caption{Cross section of the discontinuous pulse at $t=0$ and $t=1$}
 \label{fig:case3-adref-initial-final-plot}
\end{minipage}
\end{figure}

\begin{figure}
\centering
\begin{minipage}[t]{.49\textwidth}
\includegraphics[width=0.99\linewidth]{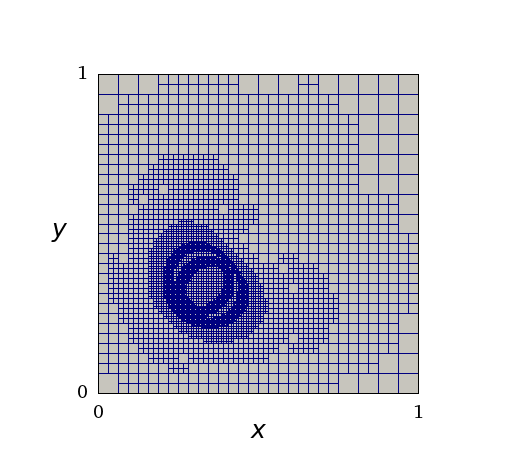}
\subcaption{An $ x-y $ cross-section of the mesh $ t = 1 $}
\label{fig:case3-adref-final-state-xy-slice}
\end{minipage}
\begin{minipage}[t]{.49\textwidth}
\includegraphics[width=0.99\linewidth]{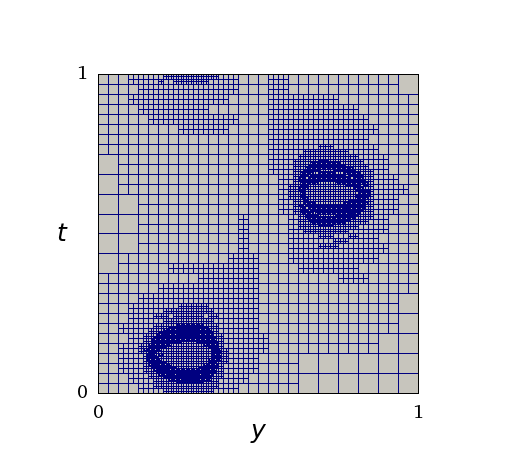}
\subcaption{A $ y-t $ cross-section of the mesh at $ x=0.5 $}
\label{fig:case3-adref-final-state-yt-slice}
\end{minipage}
\caption{Slices of the final refined mesh for the adaptive solution of \secref{subsubsec:adref-case3-sharpic}}
\label{fig:case3-adref-final-state-refined-meshes}
\end{figure}

Finally we  take a quick glance at the adaptive solution of the pure advection problem with a discontinuous initial data. This is the same problem discussed in \secref{subsubsec:case3-sharpic}. \figref{fig:case3-adref-final-state-contour-view} shows contour of the discontinuous pulse in the space-time mesh at the final refinement state. \figref{fig:case3-adref-initial-final-plot} shows the cross-section of the pulse at $t=0$ and at $t=1$. The interior of the pulse has bigger elements than the boundary of the pulse where the discontinuity lies. As mentioned before, the solution quality near the discontinuity is restricted by the underlying continuous Galerkin method which is not ideal at approximating discontinuities. This produces overshoot and undershoot near the discontinuity which results in higher residual and eventually high refinement near the discontinuity. \figref{fig:case3-adref-final-state-xy-slice} shows a 2D $x-y$ slice and \figref{fig:case3-adref-final-state-yt-slice} shows a 2D $y-t$ slice of the final refined mesh. Once again the mesh size varies greatly throughout the domain.

\section{Conclusions}\label{sec:conclusion}
In this work, we considered a coupled-space-time finite element formulation of the time-dependent advection-diffusion equation using continuous Galerkin method. Such an operator can be cast as a generalized advection-diffusion equation in the space-time domain. Due to the non-dissipative nature of the advection operator, the discrete problem corresponding to these equations may become unstable. To overcome this lack of stability of the discrete variational problem, we formulate an analogue of the Galerkin/least square (GLS) stabilization method in space-time. We show that the GLS-type regularization results in a stable discrete variational problem. We subsequently prove a priori error estimates; and also present a residual based a posteriori error estimate that is used to achieve adaptive refinement in space-time. 

We test our method on various numerical examples such as the heat equation (smooth solution) and the advection-diffusion equation  (both smooth and non-smooth solution). An interesting feature of the space-time solution is that the solution does not suffer any phase error. Also when approximating a discontinuous solution, the space-time method display considerably less oscillations compared to a sequential method of similar accuracy (Crank-Nicolson method). When coupled with an adaptive mesh refinement strategy, both smooth and discontinuous fields can be approximated closely and can even model no-loss solution in the presence of negligible diffusion.

Potential future work include extension to other stabilization techniques such as the streamline upwind Petrov Galerkin method and the variational multiscale method. In addition, nonlinear equations pose significantly different challenges compared to linear PDEs, therefore nonlinear operators such as the Cahn-Hilliard system or the Navier-Stokes equation need to be considered and analyzed with such a formulation. Furthermore, all the examples in this paper were obtained using 3D space-time (i.e., 2D + time) meshes, thus another avenue of future works include extension of the computational meshes to four-dimensions.


\section{Acknowledgements}
This work was partly supported by the National Science Foundation under the grants NSF LEAP-HI 2053760, NSF 1935255. 

\bibliographystyle{unsrt}
\bibliography{reflist}

\appendix
\section{Proofs for \secref{sec:math-formulation}}

\subsection{\lemref{lemma:boundedness} (Boundedness)}\label{sec:proof-boundedness}
\begin{proof}[Proof of \lemref{lemma:boundedness}]
	Using the triangle inequality on  \eqref{def:discrete-bilinear-form}, we have
	\begin{align*}
	|\bformh(\uh,\vh)| &= |(\uht,\vh)| + |(\adiv\uh, \vh)| + |(\visco\grad\uh,\grad\vh)| + |\innerh{\glspar \opLT \uh}{\opLT \vh}|.
	\end{align*}	
	The third and the last terms can be bounded by applying Cauchy-Schwarz inequality. Applying integration by parts on the first inner product, we have
	\begin{align*}
		(\uht,\vh) = -\inner{\uh}{\vht} + \inner{\uh}{\vh}_{\bft} - \inner{\uh}{\vh}_{\bit} = -\inner{\uh}{\vht} + \inner{\uh}{\vh}_{\bft},
	\end{align*}
	since $\vh = 0$ on $\bit$. Applying triangle inequality and Cauchy-Schwarz inequality,
	\begin{align*}
		|(\uht,\vh)| \leq |-\inner{\uh}{\vht}| + |\inner{\uh}{\vh}_{\bft} | \leq \norm{\uh}\normVh{\vh} + \normFT{\uh}\normFT{\vh}.
	\end{align*}
	The second term can be bounded by the generalized H\"{o}lder's inequality as
	\begin{align*}
		|\inner{\adiv\uh}{\vh}| \leq \normL[p]{\adv}{\stDom} \normL[q]{\grad \uh}{\stDom} \normL[r]{\vh}{\stDom},
	\end{align*}
	where $1 \leq p,q,r,\leq \infty$, and $\nicefrac{1}{p} + \nicefrac{1}{q} + \nicefrac{1}{r} = 1$. We can choose $p=4,\ q=2$, and $r=4$. Furthermore, by Sobolev embedding theorems, we have $ H^1(\stDom) \hookrightarrow L^4(\stDom)$ (for $\nsdt =2,\ 3,\ 4$). Therefore,
	\begin{align*}
		|\inner{\adiv\uh}{\vh}| &\leq \normL[4]{\adv}{\stDom} \normL[2]{\grad\uh}{\stDom} \normL[4]{\vh}{\stDom} \\
		&\leq \normH[1]{\adv}{\stDom} \normL[2]{\grad\uh}{\stDom} \normH[1]{\vh}{\stDom} \\
		&\leq \sqrt{\poincareConstant^2 +1} \normH[1]{\adv}{\stDom} \normL[2]{\grad\uh}{\stDom} \normL[2]{\grad\vh}{\stDom}, \quad \left(\text{using} \normL[2]{\cdot}{\stDom} \leq \poincareConstant \normL[2]{\grad (\cdot)}{\stDom}\right) \\
		&= \gamma \normL[2]{\grad\uh}{\stDom} \normL[2]{\grad\vh}{\stDom} \quad \left(\gamma := \sqrt{\poincareConstant^2 +1}\normH[1]{\adv}{\stDom} \right)
	\end{align*}
	where $\poincareConstant$ is Poincare's constant.
	Putting everything together and applying the generalized Cauchy-Schwarz inequality (for sums), we have
	\begin{align*}
	|\bformh(\uh,\vh)| &\leq \norm{\uh}\norm{\vht} + \normFT{\uh}\normFT{\vh} + \gamma \norm{\grad\uh} \norm{\grad \vh} + \visco\norm{\grad \uh} \norm{\grad\vh} + \normh{\glspar^{\halfnice} \opLT\uh} \normh{\glspar^{\halfnice} \opLT\vh} \\
	&\leq \left[ \norm{\uh}^2 + \normFT{\uh}^2 + \gamma \norm{\grad\uh}^2 + \visco \norm{\grad \uh}^2+ \normh{\glspar^{\halfnice} \opLT\uh}^2 \right]^{\halfnice}  \\
	&\qquad\qquad\qquad\qquad\qquad \times \left[ \norm{\vht}^2 + \normFT{\vh}^2 + \gamma \norm{\grad \vh}^2 + \visco\normh{\grad \vh}^2 + \normh{\glspar^{\halfnice} \opLT \vh}^2 \right]^{\halfnice}.
	\end{align*}
	Using Poincare's inequality $\norm{\uh} \leq \poincareConstant \norm{\grad\uh}$, 
	\begin{align*}
		|\bformh(\uh,\vh)| &\leq \left[ \normFT{\uh}^2 + (\gamma + \poincareConstant^2 + \visco) \norm{\grad\uh}^2 + \normh{\glspar^{\halfnice} \opLT\uh}^2 \right]^{\halfnice}  \\
		&\qquad\qquad \times \left[ \norm{\vht}^2 + \normFT{\vh}^2 + (\gamma + \visco) \norm{\grad \vh}^2 + \normh{\glspar^{\halfnice} \opLT \vh}^2 \right]^{\halfnice} \\
		&\leq \mu_b \normVh{\uh}\normVhstar{\vh}.
	\end{align*}
	where
	\begin{align*}
	\mu_b &= c_1 \times c_2, \\
	c_1 &=  \max \left\{ 1, \frac{(\gamma + \poincareConstant^2 + \visco)}{\visco} \right\}, \\
	c_2 &= \max \left\{ 1, \frac{(\gamma + \visco)}{\visco} \right\}.
	\end{align*}
\end{proof}

\subsection{\lemref{lemma:coercivity-stability} (Coercivity)}\label{sec:proof-coercivity}
\begin{proof}[Proof of \lemref{lemma:coercivity-stability}]
	From \eqref{def:discrete-bilinear-form}, we have
	\begin{align*}
	\bformh(\uh,\uh) &= \inner{\uht}{\uh} + \inner{\adiv \uh}{\uh} + \inner{\visco \grad \uh}{\grad \uh} +\innerh{\glspar \opLT \uh}{\opLT \uh}. \end{align*}
	The first term can be written as
	\begin{align*}
	\inner{\uht}{\uh} &= \intSpaceTime \uh \uht \  dt\ d\spDom = \half \intSpaceTime \frac{\partial}{\partial t} (\uh^2) dt d\spDom = \half \intSpace \left[ \uh^2(\xvec, T) - \uh^2(\xvec,0) \right]\ d\spDom \\
	&= \half \intSpace \uh^2(\xvec, T) \ d\spDom \\
	&= \halfnice \norm{\uh}_{\bft}^2.
	\end{align*}
	Using $ \divergence \adv = 0 $, the second term can be written as
	\begin{align*}
		(\adiv\uh, \uh) = (\adiv\uh, \uh) + \inner{(\divergence \adv)\uh}{\uh} = \inner{\divergence (\uh \adv^T)}{\uh} &= -\inner{\uh \adv^T}{\grad \uh} = -\inner{\uh}{\adiv \uh}.
	\end{align*}
	Which implies $ 2\inner{\adiv \uh}{\uh}  = 0 $. So, we have
	\begin{align*}
	\bformh(\uh,\uh) &= \halfnice \norm{\uh}_{\bft}^2 + \visco \norm{\grad \uh}^2 +\normh{\glspar^{\halfnice} \opLT \uh}^2 \\
	&\geq \mu_c\ \normVh{\uh}^2,
	\end{align*}
	where $ \mu_c = \halfnice$.
\end{proof}

\subsection{\cororef{lemma:interp-boundary-estimate}} \label{sec:proof-interp-boundary}
\begin{proof}[Proof of \cororef{lemma:interp-boundary-estimate}]
	We make use of the trace inequality (\cite[Sec 1.4.3]{di2011mathematical}, and \cite[Sec 3.3]{verfurth2013posteriori}) defined for each element $K$ in the mesh,
	\begin{align*}
	\normL{v-\interp{v}}{\partial K}^2 &\leq C_E \left( h_K^{-1}\normL{v-\interp{v}}{K}^2 + h_K\normL{\grad(v-\interp{v})}{K}^2 \right) \\
	&\leq  C_E \left( h_K^{-1}{\approxConstant}_0^2 \he^{2(k+1-0)}|v|_{H^{k+1}(K)}^2 + h_K {\approxConstant}_1^2 \he^{2(k+1-1)}|v|_{H^{k+1}(K)}^2 \right) \\	
	&\leq C_E ({\approxConstant}_0^2 + {\approxConstant}_1^2) \he^{2k+1} |v|_{H^{k+1}(K)}^2 \\
	&= C_{aE}^2 \he^{2k+1} |v|_{H^{k+1}(K)}^2,
	\end{align*}
	where $ C_{aE} = \sqrt{C_E ({\approxConstant}_0^2 + {\approxConstant}_1^2)} $.
\end{proof}

\subsection{\lemref{lemma:interp-estimate-in-Vh}} \label{proof:interp-estimate-vh}
\begin{proof}[Proof of \lemref{lemma:interp-estimate-in-Vh}]
	Suppose, $ \ei = \interp{u} - u $, and denote $ \adivst = \partial_t + \adiv $ Then we have
	\begin{align*}
	\normVh{\ei}^2 &= \left[\normVExpanded{\ei}\right] \\
	&\leq {\approxConstant}_{\Gamma}^2 h^{2k+1}\seminormH[k+1]{u}{\stDom}^2  + \visco {\approxConstant}_1^2 h^{2k} \seminormH[k+1]{u}{\stDom}^2 + \elmsum \glspar_K \normElm{\adivst{\ei} - \visco\laplacian\ei}^2 \\
	&\leq \left({\approxConstant}_{\Gamma}^2 h^{2k+1}  + \visco {\approxConstant}_1^2 h^{2k}\right)\seminormH[k+1]{u}{\stDom}^2 + \elmsum \glspar_K \left[ \normElm{\adivst \ei} + \visco\normElm{\laplacian \ei} \right]^2 \\
	&\leq \left({\approxConstant}_{\Gamma}^2 h^{2k+1}  + \visco {\approxConstant}_1^2 h^{2k}\right)\seminormH[k+1]{u}{\stDom}^2 + \elmsum \glspar_K \left[ {\approxConstant}_1\normL[2]{\advst}{K} \he^{k} \seminormH[k+1]{u}{K} + \visco{\approxConstant}_{2} \he^{k-1} \seminormH[k+1]{u}{K} \right]^2 \\
	&\leq \left({\approxConstant}_{\Gamma}^2 h^{2k+1}  + \visco {\approxConstant}_1^2 h^{2k}\right)\seminormH[k+1]{u}{\stDom}^2 + \elmsum \glspar_K \left[ {\approxConstant}_1 \normL[2]{\advst}{K} \he + \visco{\approxConstant}_{2} \right]^2 \he^{2(k-1)} \seminormH[k+1]{u}{K}^2
	\end{align*}
	Now, $\glspar_K$ is given by (see \eqref{eq:glspar-formula})
	\begin{align*}
		\glspar_K = \left[ \frac{c_1 \normL[2]{\advst}{K}}{\he}  + \frac{c_2 \visco}{\he^2} \right]^{-1} = \frac{\he^2}{c_1 \normL[2]{\advst}{K} \he + c_2 \visco}
	\end{align*}
	So, we have
	\begin{align*}
		\glspar_K \left[ {\approxConstant}_1 \normL[2]{\advst}{K} \he + \visco{\approxConstant}_{2} \right]^2 = \he^2 \frac{\left[ {\approxConstant}_1 \normL[2]{\advst}{K} \he + \visco{\approxConstant}_{2} \right]^2}{c_1 \normL[2]{\advst}{K} \he + c_2 \visco} \leq C \he^2.
	\end{align*}
	Substituting,
	\begin{align*}
		\normVh{\ei}^2 &\leq \left({\approxConstant}_{\Gamma}^2 h^{2k+1}  + {\approxConstant}_1^2 h^{2k}\right)\seminormH[k+1]{u}{\stDom}^2 + \elmsum C \he^{2k} \seminormH[k+1]{u}{K}^2 \\
		&\leq \left({\approxConstant}_{\Gamma}^2 h^{2k+1}  + {\approxConstant}_1^2 h^{2k} + Ch^{2k}\right)\seminormH[k+1]{u}{\stDom}^2 \\
		&\leq \left({\approxConstant}_{\Gamma}^2 h  + {\approxConstant}_1^2 + C\right) h^{2k} \seminormH[k+1]{u}{\stDom}^2 \\
		&\leq C^2 h^{2k} \seminormH[k+1]{u}{\stDom}^2.
	\end{align*}
	Thus, we have the required estimate.
\end{proof}

\subsection{\thmref{theorem:a-priori-estimate} (A priori error estimate)} \label{sec:proof-a-priori}
\begin{proof}[Proof of \thmref{theorem:a-priori-estimate}]
	Define $\er = \uh - u$, $\eh = \uh-\interpolant u$ and $\ei = \interpolant u - u$. These three quantities are related as
	\begin{align}
	\er &=  \uh - u \\ &= \uh - \interp{u} + \interp{u} - u \\ &= \eh + \ei
	\end{align}
	By triangle inequality
	\begin{align}
	\normVh{\er} \leq \normVh{\eh} + \normVh{\ei}
	\label{eq:errors-triangle-ineq}
	\end{align}
	From \lemref{lemma:interp-estimate-in-Vh}, the estimate on $\normVh{\ei}$ is already known, i.e.,
	\begin{align}
	\normVh{\ei}\leq C h^k|v|_{H^{k+1}(\stDom)}.
	\end{align}
	We can similarly show that
	\begin{align}
		\normVhstar{\ei}\leq C h^k|v|_{H^{k+1}(\stDom)}.
	\end{align}
	Now we try to estimate $\normVh{\eh}$ using $ \bformh(\eh, \eh) $. We have
    \begin{align*}
			\mu_c \normVh{\eh}^2 
			&\leq \bformh(\eh,\eh) = \bformh(\er-\ei,\eh) = \bformh(\er,\eh)-\bform(\ei,\eh) = -\bform(\ei,\eh) \leq |\bform(\ei,\eh)| \leq \mu_b \normVh{\ei} \normVhstar{\eh} \\
			&\leq \mu_b C h^{2k}\seminormH[k+1]{v}{\stDom}^2,
	\end{align*}
    i.e,
    \begin{align*}
        \normVh{\eh} &\leq C h^{k}\seminormH[k+1]{v}{\stDom}.
    \end{align*}
	Substituting in ~\eqref{eq:errors-triangle-ineq}, we get
	\begin{align}
	\normVh{\er} 
	&\leq \normVh{\eh} + \normVh{\ei}\\
	&\leq (C_1 + C_2)h^{k}\seminormH[k+1]{v}{\stDom}\\
	&\leq C h^{k}\seminormH[k+1]{v}{\stDom}.
	\end{align}
\end{proof}

\subsection{\thmref{theorem:a-posteriori-estimate} (A posteriori error estimate)} \label{sec:proof-a-posteriori}
\begin{proof}[Proof of \thmref{theorem:a-posteriori-estimate}]
	The proof below closely follows the arguments presented in Section 1.4 of \cite{verfurth2013posteriori}.
	We can rewrite ~\eqref{eq:gls-abstract-for-exact-solution} and \eqref{eq:exact-solution-discrete-form} for any $ v \in \spaceV $ as
	\begin{subequations}
		\begin{align}
		\innerh{\opLT u - F}{v + \glspar\opLT v} &= 0 \ \ \forall v \in \spaceV, \\
		\implies \bformh(u, v) &= \lformh(v) \ \ \forall v \in \spaceV.
		\end{align}
	\end{subequations}
	Let us denote the strong residual as $ R(u) = \opLT u - F $, and the weak residual as
	\begin{align}
	\innerh{R(\uh)}{v} &= \lformh(v) - \bformh(\uh,v) \\
	&= \bformh(u,v) - \bformh(\uh,v) \\
	&= \bformh(u-\uh,v).
	\end{align}
	So, we have
	\begin{align}
	\innerh{R(\uh)}{v} &= \lformh(v) - \bformh(\uh,v) \\
	&= \innerh{F}{v + \glspar \opLT v} - \left[ (\uht, v) + (\adiv\uh,  v) + (\visco\grad\uh,\grad v) +\innerh{\glspar \opLT \uh}{\opLT  v} \right] \\
	&= \elmsum \inner{F}{v + \glspar \opLT v}_K - \left[ (\uht, v)_K + (\adiv\uh,  v)_K + (\visco\grad\uh,\grad v)_K +\inner{\opLT \uh}{\glspar \opLT  v}_K \right] \\
	&= \elmsum \inner{F}{v + \glspar \opLT v}_K \\ &\qquad\qquad - \left[ (\uht, v)_K + (\adiv\uh,  v)_K - (\visco\laplacian\uh,v)_K + \Xsum_{E \in \partial K} \inner{\visco \ddiv{\normalvec}}{v} +\inner{\opLT \uh}{\glspar \opLT  v}_K \right] \\
	&= \elmsum \inner{F}{v + \glspar \opLT v}_K - \left[ \inner{\opLT \uh}{v}_K + \inner{\opLT \uh}{\glspar \opLT  v}_K + \Xsum_{E \in \partial K} \inner{\visco \ddiv{\normalvec}}{v} \right] \\
	&= \elmsum \inner{F - \opLT \uh}{v + \glspar \opLT v}_K - \edgesum \inner{\visco \jump(\ddiv{\normalvec}\uh)}{v}_E.
	\end{align}
	Define the elemental residual as $ r = (F - \opLT \uh) |_K $ and the jump on an edge as $ j = -(\visco \jump(\ddiv{\normalvec}\uh))|_E$. Then we have
	\begin{align} \label{eq:aposter-weak-res-v}
	\innerh{R(\uh)}{v} &= \elmsum \left[\inner{r}{v}_K + \inner{r}{\glspar \opLT v}_K\right] + \edgesum \inner{j}{v}_E \ \text{for any}\ v \in \spaceV
	\end{align}
	Now, by Galerkin orthogonality, we have
	\begin{align}
	\innerh{R(\uh)}{\vh} = 0 \ \text{for any}\ \vh \in \spaceVh.
	\end{align}
	So, using $ \vh \in \spaceVh \subset \spaceV $ in \eqref{eq:aposter-weak-res-v}, we have
	\begin{align} \label{eq:aposter-weak-res-vh}
	0 &= \elmsum \left[\inner{r}{\vh}_K + \inner{r}{\glspar \opLT \vh}_K\right] + \edgesum \inner{j}{\vh}_E \ \forall \vh \in \spaceVh.
	\end{align}
	Subtracting \eqref{eq:aposter-weak-res-vh} from \eqref{eq:aposter-weak-res-v}, we have for any $ v \in \spaceV $ and every $ \vh \in \spaceVh $,
	\begin{align} 
	\innerh{R(\uh)}{v} &= \elmsum \left[\inner{r}{v-\vh}_K + \inner{r}{\glspar \opLT (v-\vh)}_K\right] + \edgesum \inner{j}{(v-\vh)}_E \\
	&\leq \elmsum \normElm{r} \left(\normElm{v-\vh} + \glsparMax \normElm{\opLT (v-\vh)} \right) + \edgesum \norm{j}_E \norm{v-\vh}_E.
	\end{align}
	Now, choosing $ \vh = \interp{v} $ and using the estimates from \eqref{estimate:sobolev-interp}, we have
	\begin{align}
	\innerh{R(\uh)}{v} &\leq \elmsum \normElm{r} \left( C_{a0}\he^{k+1} + \glsparMax C_{a2} \he^{k-1} \right) \seminormH[k+1]{v}{K} + \edgesum \norm{j}_E C_{a \Gamma} \he^{k+\half}  \seminormH[k+1]{v}{E}.
	\end{align}
	Assuming $ \glsparMax = \gamma h^2,\ \gamma>0 $, (i.e., assuming a diffusive limit, see \remref{remark:stabilization}), we have
	\begin{align}
	\innerh{R(\uh)}{v} &\leq \elmsum \he^{k+1} \normElm{r} \left( C_{a0} + \gamma C_{a2} \right) \seminormH[k+1]{v}{K} + \edgesum \norm{j}_E C_{a \Gamma} h_E^{k+\half}  \seminormH[k+1]{v}{E}.
	\end{align}
	Using Cauchy-Schwarz inequality, we have
	\begin{align}
	\innerh{R(\uh)}{v} &\leq \ C \  \left[\elmsum \he^{2(k+1)} \normElm{r}^2 + 
	\edgesum h_E^{2k+1} \norm{j}_E^2 \right]^{\halfnice} \\
	&\qquad\qquad \times \left[\elmsum \seminormH[k+1]{v}{K}^2 + \edgesum \seminormH[k+1]{v}{E}^2 \right]^{\halfnice}.
	\end{align}
	The first term in the brackets on the right hand side can be bounded as
	\begin{align*}
	\left[\elmsum \he^{2(k+1)} \normElm{r}^2 + 
	\edgesum h_E^{2k+1} \norm{j}_E^2 \right]^{\halfnice} &\leq \left[\elmsum \left(\he^{2(k+1)} \normElm{r}^2 + 
	\Xsum_{E \in \mathcal{E}_K} h_E^{2k+1} \norm{j}_E^2\right) \right]^{\halfnice} \\
	&\leq \left[\elmsum h^{2k} \left(\he^{2} \normElm{r}^2 + 
	\Xsum_{E \in \mathcal{E}_K} h_E \norm{j}_E^2\right) \right]^{\halfnice} \\
	&\leq h^{k} \left(\elmsum {\etaElm}^2\right)^{\halfnice}, \\
	&= h^k \etaFull
	\end{align*}
	where $ h $ is defined as
	\begin{align}
	h = \max \left\{ \max_{K \in \mesh} h_K, \ \ \max_{E \in \edgeset} h_E  \right\}.
	\end{align}
	Also, due to shape regularity, we can write
	\begin{align}
	\left[\elmsum \seminormH[k+1]{v}{K}^2 + \edgesum \seminormH[k+1]{v}{E}^2 \right]^{\halfnice} \leq C \seminormH[k+1]{v}{\stDom}.
	\end{align}
	Thus, we have
	\begin{align}
	\innerh{R(\uh)}{v} 
	&\leq C \etaFull h^k \seminormH[k+1]{v}{\stDom}.
	\end{align}
	Now, using $ \innerh{R(\uh)}{v} = \bformh(u-\uh, v) $, we also have
	\begin{align}
	\innerh{R(\uh)}{v} &= \bformh(u-\uh, v) \\
	&=  \bformh(u-\uh, v) -  \bformh(u-\uh, \vh) \quad \text{(by Galerkin orthogonality)} \\
	&= \bformh(u-\uh, v-\vh) \\
	&\leq \mu_b \normVhstar{u-\uh}\normVh{v-\vh}.
	\end{align}
	Once again, choosing $ \vh = \interp{v} $, we have
	\begin{align}
	\innerh{R(\uh)}{v} &\leq \mu_b \normVhstar{u-\uh} C_{a0} h^k \seminormH[k+1]{v}{\stDom} \\
	\implies \oneOver{\mu_b C_{a0} h^k \seminormH[k+1]{v}{\stDom}} \innerh{R(\uh)}{v} &\leq \normVhstar{u-\uh}.
	\end{align}
	Using a similar argument as in \cite{verfurth2013posteriori} (Theorem 1.5), we can say
	\begin{align}
	\normVhstar{u-\uh} &= \sup_{v \in \spaceV \backslash \{0\}} \oneOver{\mu_b C_{a0} h^k \seminormH[k+1]{v}{\stDom}} \innerh{R(\uh)}{v} \\
	&\leq \oneOver{\mu_b C_{a0} h^k \seminormH[k+1]{v}{\stDom}} C \etaFull h^k \seminormH[k+1]{v}{\stDom} \\
	&\leq C\etaFull.
	\end{align}
\end{proof}

\section{Additional results}
\subsection{Study on condition numbers}
Below, we present a brief study on how the the condition number of the global coefficient matrix $\globalMatrix$ varies with respect to presence of the proposed stabilization and also with respect to the usage of a preconditioner in the linear algebra solution algorithm. The following results were obtained with the generalized minimal residual algorithm (GMRES) provided by \petsc{}. In the tables below, ``No PC'' means that no preconditioner was used with GMRES; whereas ASM refers to the fact that the additive Schwarz method was used with GMRES.

\begin{table}[]
	\centering
	\caption{Condition numbers for \textbf{\textit{heat equation}} with trilinear basis functions ($m=1$)}
	\label{tab:condition-heat-lin}
	\begin{tabular}{@{}c|cccc|cccc@{}}
		\toprule[2pt]
		\multirow{3}{*}{$h$} & \multicolumn{4}{c|}{$\nu = 10^{-2}$} & \multicolumn{4}{c}{$\nu = 10^{-6}$} \\ \cmidrule(l){2-9} 
		& \multicolumn{2}{c|}{No stabilization} & \multicolumn{2}{c|}{GLS stabilization} & \multicolumn{2}{c|}{No stabilization} & \multicolumn{2}{c}{GLS stabilization} \\ \cmidrule(l){2-9} 
		& \multicolumn{1}{c|}{No PC} & \multicolumn{1}{c|}{ASM} & \multicolumn{1}{c|}{No PC} & ASM & \multicolumn{1}{c|}{No PC} & \multicolumn{1}{c|}{ASM} & \multicolumn{1}{c|}{No PC} & ASM \\ \midrule[2pt]
		$\nicefrac{1}{8}$ & \multicolumn{1}{c|}{7.2} & \multicolumn{1}{c|}{$ 16.2 $ } & \multicolumn{1}{c|}{5.2} & 4.4 & \multicolumn{1}{c|}{10.7} & \multicolumn{1}{c|}{$ 3.0 \times 10^{7} $ } & \multicolumn{1}{c|}{7.8} & 6.4 \\ \midrule
		$\nicefrac{1}{16}$ & \multicolumn{1}{c|}{14.3} & \multicolumn{1}{c|}{$ 9.2 $ } & \multicolumn{1}{c|}{9.6} & 4.0 & \multicolumn{1}{c|}{20.9} & \multicolumn{1}{c|}{$ 7.9 \times 10^{5} $ } & \multicolumn{1}{c|}{15.0} & 5.5 \\ \midrule
		$\nicefrac{1}{32}$ & \multicolumn{1}{c|}{36.0} & \multicolumn{1}{c|}{$ 10.5 $ } & \multicolumn{1}{c|}{17.8} & 6.3 & \multicolumn{1}{c|}{41.3} & \multicolumn{1}{c|}{$ 5.5 \times 10^{8} $ } & \multicolumn{1}{c|}{29.4} & 5.6 \\ \bottomrule[2pt]
	\end{tabular}
\end{table}

\begin{table}[]
	\centering
	\caption{Condition numbers for \textbf{\textit{heat equation}} with triquadratic basis functions ($m=2$)}
	\label{tab:condition-heat-qua}
	\begin{tabular}{@{}c|cccc|cccc@{}}
		\toprule[2pt]
		\multirow{3}{*}{$h$} & \multicolumn{4}{c|}{$\nu = 10^{-2}$} & \multicolumn{4}{c}{$\nu = 10^{-6}$} \\ \cmidrule(l){2-9} 
		& \multicolumn{2}{c|}{No stabilization} & \multicolumn{2}{c|}{GLS stabilization} & \multicolumn{2}{c|}{No stabilization} & \multicolumn{2}{c}{GLS stabilization} \\ \cmidrule(l){2-9} 
		& \multicolumn{1}{c|}{No PC} & \multicolumn{1}{c|}{ASM} & \multicolumn{1}{c|}{No PC} & ASM & \multicolumn{1}{c|}{No PC} & \multicolumn{1}{c|}{ASM} & \multicolumn{1}{c|}{No PC} & ASM \\ \midrule[2pt]
		$\nicefrac{1}{8}$ & \multicolumn{1}{c|}{19.2} & \multicolumn{1}{c|}{$ 2.0 \times 10^{2} $ } & \multicolumn{1}{c|}{38.6} & 9.1 & \multicolumn{1}{c|}{196.4} & \multicolumn{1}{c|}{$ 5.0 \times 10^{5} $ } & \multicolumn{1}{c|}{379.3} & 17.4 \\ \midrule
		$\nicefrac{1}{16}$ & \multicolumn{1}{c|}{52.3} & \multicolumn{1}{c|}{$ 9.5 \times 10^{2} $ } & \multicolumn{1}{c|}{106.2} & 16.0 & \multicolumn{1}{c|}{477.2} & \multicolumn{1}{c|}{$ 2.5 \times 10^{5} $ } & \multicolumn{1}{c|}{914.1} & 33.7 \\ \midrule
		$\nicefrac{1}{32}$ & \multicolumn{1}{c|}{202.0} & \multicolumn{1}{c|}{$ 1.5 \times 10^{2} $ } & \multicolumn{1}{c|}{477.0} & 81.0 & \multicolumn{1}{c|}{524.2} & \multicolumn{1}{c|}{$ 1.2 \times 10^{9} $ } & \multicolumn{1}{c|}{$ 1.8 \times 10^{3} $} & 67.4 \\ \bottomrule[2pt]
	\end{tabular}
\end{table}

\begin{table}[]
	\centering
	\caption{Condition numbers for \textbf{\textit{advection-diffusion equation}} with trilinear basis functions ($m=1$)}
	\label{tab:condition-advdif-lin}
	\begin{tabular}{@{}c|cccc|cccc@{}}
		\toprule[2pt]
		\multirow{3}{*}{$h$} & \multicolumn{4}{c|}{$\nu = 10^{-2}$} & \multicolumn{4}{c}{$\nu = 10^{-6}$} \\ \cmidrule(l){2-9} 
		& \multicolumn{2}{c|}{No stabilization} & \multicolumn{2}{c|}{GLS stabilization} & \multicolumn{2}{c|}{No stabilization} & \multicolumn{2}{c}{GLS stabilization} \\ \cmidrule(l){2-9} 
		& \multicolumn{1}{c|}{No PC} & \multicolumn{1}{c|}{ASM} & \multicolumn{1}{c|}{No PC} & ASM & \multicolumn{1}{c|}{No PC} & \multicolumn{1}{c|}{ASM} & \multicolumn{1}{c|}{No PC} & ASM \\ \midrule[2pt]
		$\nicefrac{1}{8}$ & \multicolumn{1}{c|}{31.5} & \multicolumn{1}{c|}{$ 2.8 \times 10^{6} $ } & \multicolumn{1}{c|}{12.4} & 3.8 & \multicolumn{1}{c|}{116.4} & \multicolumn{1}{c|}{$ 8.7 \times 10^{8} $ } & \multicolumn{1}{c|}{15.3} & 4.8 \\ \midrule
		$\nicefrac{1}{16}$ & \multicolumn{1}{c|}{42.0} & \multicolumn{1}{c|}{$ 13.0 $ } & \multicolumn{1}{c|}{32.3} & 5.8 & \multicolumn{1}{c|}{523.3} & \multicolumn{1}{c|}{$ 4.8 \times 10^{9} $ } & \multicolumn{1}{c|}{43.3} & 7.8 \\ \midrule
		$\nicefrac{1}{32}$ & \multicolumn{1}{c|}{95.3} & \multicolumn{1}{c|}{$ 20.3 $ } & \multicolumn{1}{c|}{71.5} & 10.4 & \multicolumn{1}{c|}{783.6} & \multicolumn{1}{c|}{$ 1.1 \times 10^{11} $ } & \multicolumn{1}{c|}{96.1} & 14.0 \\ \bottomrule[2pt]
	\end{tabular}
\end{table}

\begin{table}[]
	\centering
	\caption{Condition numbers for \textbf{\textit{advection-diffusion equation}} with triquadratic basis functions ($m=2$)}
	\label{tab:condition-advdif-qua}
	\begin{tabular}{@{}c|cccc|cccc@{}}
		\toprule[2pt]
		\multirow{3}{*}{$h$} & \multicolumn{4}{c|}{$\nu = 10^{-2}$} & \multicolumn{4}{c}{$\nu = 10^{-6}$} \\ \cmidrule(l){2-9} 
		& \multicolumn{2}{c|}{No stabilization} & \multicolumn{2}{c|}{GLS stabilization} & \multicolumn{2}{c|}{No stabilization} & \multicolumn{2}{c}{GLS stabilization} \\ \cmidrule(l){2-9} 
		& \multicolumn{1}{c|}{No PC} & \multicolumn{1}{c|}{ASM} & \multicolumn{1}{c|}{No PC} & ASM & \multicolumn{1}{c|}{No PC} & \multicolumn{1}{c|}{ASM} & \multicolumn{1}{c|}{No PC} & ASM \\ \midrule[2pt]
		$\nicefrac{1}{8}$ & \multicolumn{1}{c|}{59.1} & \multicolumn{1}{c|}{$ 7.3 \times 10^{9} $ } & \multicolumn{1}{c|}{114.7} & 15.7 & \multicolumn{1}{c|}{311.8} & \multicolumn{1}{c|}{$ 1.0 \times 10^{9} $ } & \multicolumn{1}{c|}{153.7} & 23.7 \\ \midrule
		$\nicefrac{1}{16}$ & \multicolumn{1}{c|}{146.4} & \multicolumn{1}{c|}{$ 77.1 $ } & \multicolumn{1}{c|}{264.9} & 33.0 & \multicolumn{1}{c|}{955.6} & \multicolumn{1}{c|}{$ 1.4 \times 10^{8} $ } & \multicolumn{1}{c|}{334.7} & 48.6 \\ \midrule
		$\nicefrac{1}{32}$ & \multicolumn{1}{c|}{274.9} & \multicolumn{1}{c|}{$ 92.2 $ } & \multicolumn{1}{c|}{651.8} & 85.0 & \multicolumn{1}{c|}{609.7} & \multicolumn{1}{c|}{$ 7.6 \times 10^{11} $ } & \multicolumn{1}{c|}{738.3} & 100.6 \\ \bottomrule[2pt]
	\end{tabular}
\end{table}

\clearpage

\subsection{Convergence studies with higher order basis functions}
\label{sec:conv-study-higher-order}
\input{sections/7_convergence_new.tex}
\end{document}